\documentclass[a4paper, 12pt]{amsart}
\usepackage{amsfonts, amsmath, amssymb, amsthm, enumitem}  
\usepackage[english]{babel}
\usepackage[utf8]{inputenc}
\usepackage[OT1]{fontenc}
\usepackage{bm}

\usepackage[margin=2.2cm]{geometry}

\usepackage{textcomp, cmap, comment}	
\usepackage{graphicx, wrapfig}
\usepackage{epigraph, xcolor, hyperref}
\usepackage{array,tabularx,tabulary,booktabs}
\usepackage{euscript, mathrsfs}

\newcounter{iii}

\newcommand{\bb}{{\mathcal B}}
\newcommand{\aaa}{{\mathcal A}}
\newcommand{\mm}{\mathcal M}
\newcommand{\G}{\mathcal G}

\newcommand{\ff}{\mathcal F}
\theoremstyle{plain}
\newtheorem{thm}{Theorem}

\newtheorem{lem}[thm]{Lemma}
\newtheorem{prop}[thm]{Proposition}
\newtheorem*{obs}{Observation}
\newtheorem{pro}{Problem}
\newtheorem{cor}[thm]{Corollary}
\theoremstyle{definition}
\newtheorem{defi}{Definition}

\numberwithin{equation}{section}
\numberwithin{thm}{section}

\title{Structure and properties of large intersecting families}
\author{Andrey Kupavskii}
\address{University of Birmingham,
Moscow Institute of Physics and Technology; Email: {\tt kupavskii@ya.ru}.} \thanks{The research was supported by the Advanced Postdoc.Mobility grant no. P300P2\_177839 of the Swiss National Science Foundation.}
\date{}

\begin{document}
\maketitle
\begin{abstract} We say that a family of $k$-subsets of an $n$-element set is {\it intersecting}, if any two of its sets intersect. In this paper we study properties and structure of large intersecting families.

We prove a conclusive version of Frankl's theorem on intersecting families with bounded maximal degree. This theorem, along with its generalizations to cross-intersecting families, strengthens the results obtained by Frankl, Frankl and Tokushige, Kupavskii and Zakharov and others.

We study the structure of large intersecting families, obtaining some very general structural theorems which extend the results of Han and Kohayakawa, as well as Kostochka and Mubayi.

We also obtain an extension of some classic problems on intersecting families introduced in the 70s. We extend an old result of Frankl, in which he determined the size and structure of the largest intersecting family of $k$-sets with covering number $3$ for $n>n_0(k)$. We obtain the same result for $n>Ck$, where $C$ is an absolute constant. Finally, we obtain a similar extension for the following problem of Erd\H os, Rothschild and Sz\'emeredi: what is the largest intersecting family, in which no element is contained in more than a $c$-proportion of the sets, for different values of $c$.
\end{abstract}

\section{Introduction}

For integers $a\le b$,  put $[a,b]:=\{a,a+1,\ldots, b\}$, and denote $[n]:=[1,n]$ for shorthand. For a set $X$, denote by $2^{X}$ its power set and, for integer $k\ge 0$,  denote by ${X\choose k}$ the collection of all $k$-element subsets ({\it $k$-sets}) of $X$. A \underline{family} is simply a collection of sets.  We call a family \underline{intersecting}, if any two of its sets intersect. A ``trivial'' example of an intersecting family is the family of all sets containing a fixed element. We call a family \underline{non-trivial}, if the intersection of all sets from the family is empty.

One of the oldest and most famous results in extremal combinatorics is the Erd\H os--Ko--Rado theorem:
\begin{thm}[\cite{EKR}] Let $n\ge 2k>0$ and consider an intersecting family $\ff\subset {[n]\choose k}$. Then $|\ff|\le {n-1\choose k-1}$. Moreover, for $n>2k$ the equality holds only for the families of all $k$-sets containing a given element. \end{thm}

Answering a question of Erd\H os, Ko, and Rado, Hilton and Milner \cite{HM} found the size and structure of the largest non-trivial intersecting families of $k$-sets. For $k\ge 4$, up to a permutation of the ground set, it must have the form $\mathcal H_{k}$, where for integer $2\le u\le k$
\begin{equation}\label{eqhu}\mathcal H_u:=\ \Big\{A\in {[n]\choose k}\ :\ [2,u+1]\subset A\Big\}\cup\Big\{A\in{[n]\choose k}\ :\  1\in A, [2,u+1]\cap A\ne \emptyset\Big\}.\end{equation}
$\mathcal J_1$ has size ${n-1\choose k-1}-{n-k-1\choose k-1}+1$, which is much smaller than ${n-1\choose k-1}$, provided $n$ is large as compared to $k$.

For a family $\ff\subset 2^{[n]}$ and $i\in [n]$, the \underline{degree} $d_i(\ff)$ of $i$ in $\ff$ is the number of sets from $\ff$ containing $i$. Let $\Delta(\ff)$ stand for the \underline{maximal degree} of an element in $\ff$. Frankl \cite{Fra1} proved the following far-reaching generalization of the Hilton--Milner theorem.

\begin{thm}[\cite{Fra1}]\label{thmfr} Let $n>2k>0$ and $\ff\subset {[n]\choose k}$ be an intersecting family. If $\Delta(\ff)\le {n-1\choose k-1}-{n-u-1\choose k-1}$ for some integer $3\le u\le k$, then \begin{equation*}|\ff|\le {n-1\choose k-1}+{n-u-1\choose n-k-1}-{n-u-1\choose k-1}.\end{equation*}\end{thm}

One can deduce the Hilton--Milner theorem from the $u=k$ case of Theorem~\ref{thmfr}. Theorem~\ref{thmfr} is sharp for integer values of $u$, as witnessed by the example \eqref{eqhu}. On a high level, it provides us with an upper bound on $|\ff|$ in terms of the size of the largest trivial subfamily ({\it star}) in $\ff$. Let us state a stronger version of Theorem~\ref{thmfr} in dual terms.
For a family $\ff$, the \underline{diversity} $\gamma(\ff)$ is the quantity $|\ff|-\Delta(\ff)$. One may think of diversity as of the distance from $\ff$ to the closest star.

The following strengthening of Theorem~\ref{thmfr} was obtained by Kupavskii and Zakharov \cite{KZ}.
\begin{thm}[\cite{KZ}]\label{thm1} Let $n>2k>0$ and $\ff\subset {[n]\choose k}$ be an intersecting family. If $\gamma(\ff)\ge {n-u-1\choose n-k-1}$ for some real $3\le u\le k$, then \begin{equation}\label{eq01}|\ff|\le {n-1\choose k-1}+{n-u-1\choose n-k-1}-{n-u-1\choose k-1}.\end{equation}
\end{thm}
We note that the Hilton--Milner theorem, as well as Theorem~\ref{thmfr}, is immediately implied by Theorem~\ref{thm1}. The deduction of Theorem~\ref{thm1} from Theorem~\ref{thmfr} for integer values of $u$ is possible, but not straightforward. Theorem~\ref{thm1} provides the strongest known stability result for the Erd\H os--Ko--Rado theorem for large intersecting families, more precisely, for the families of size at least ${n-2\choose k-2}+2{n-3\choose k-2}$. There are several other stability results for the Erd\H os--Ko--Rado theorem, see, e.g. \cite{DT,EKL,Fri}.

In Section~\ref{sec3}, we prove a conclusive version of Theorem~\ref{thm1}, which gives the precise dependence of size of an intersecting family $\ff$ of $k$-sets on (the lower bound on) $\gamma(\ff)$. The result then is extended to cover the equality case, as well as the weighted case and the case of cross-intersecting families. In particular, it strengthens the results of \cite{Fra1}, \cite{FT}, \cite{KZ}.
 One of the main ingredients in the proof of the main result of Section~\ref{sec3} (as well as in the proofs of Theorems~\ref{thmfr},~\ref{thm1}) is the famous Kruskal--Katona theorem \cite{Kr,Ka}. This theorem is a central result in extremal set theory and has numerous applications (see, e.g., \cite{BK, BT}). Another key ingredient is the bipartite switching trick, which was introduced in \cite{KZ} (similar ideas appeared earlier in \cite{FK1}). In this paper, we exploit this trick to a much greater extent.

 We do not state the aforementioned theorem in this section since it requires some preparations. Instead, we state the following corollary, which can be seen as a generalization of the Hilton--Milner phenomena.

\begin{cor}\label{corhm}
Let $n>2k\ge 8$ and $\ff\subset {[n]\choose k}$ be an intersecting family. Suppose that $\gamma(\ff)>{n-u-1\choose n-k-1}$ for some integer $4\le u\le k$. Then
\begin{equation}\label{eqhm}|\ff|\le {n-1\choose k-1}+{n-u-1\choose n-k-1}-{n-u-1\choose k-1}-{n-k-2\choose k-2}+1.\end{equation}
Moreover, the same inequality holds for $n>2k\ge 6$ with $u=3$ if ${n-4\choose k-3}<\gamma(\ff)\le {n-3\choose k-2}-{n-k-2\choose k-2}+1$ or ${n-3\choose k-2}< \gamma(\ff)\le {n-3\choose k-2}+{n-4\choose k-2}-{n-k-2\choose k-2}+1$.
\end{cor}
Compare \eqref{eqhm} with \eqref{eq01} for integer $u$. The difference in the bounds is ${n-k-2\choose k-2}-1$, while the lower bounds on diversity differ by $1$. Thus, Corollary~\ref{corhm} states that the size of the largest family with diversity at least $\gamma$ has a big drop when $\gamma$ passes the point ${n-u-1\choose n-k-1}$ for integer $u$. We also note that Corollary~\ref{corhm} is sharp, as it will be clear from Section~\ref{sec3}.\\

Numerous authors aimed to determine {\it precisely}, what are the largest intersecting families in ${[n]\choose k}$ with certain restrictions. One of such questions was studied by Han and Kohayakawa, who determined the largest family of intersecting families that is neither contained in the Erd\H os--Ko--Rado family, nor in the Hilton--Milner family. In our terms, the question can be simply restated as follows: what is the largest intersecting family with $\gamma(\ff)\ge 2$? The proof of Han and Kohayakawa is quite technical and long. Kruskal--Katona-based arguments allow for a very short and simple proof in the case $k\ge 5$. For $i\in[k]$ let us put $I_i:=[i+1,k+i]$ and
$$\mathcal J_i:=\ \{I_1,I_i\}\cup \Big\{F\in{[n]\choose k}\ :\ 1\in F, F\cap I_1\ne \emptyset, F\cap I_i\ne \emptyset\Big\}.$$
We note that $\mathcal J_i\subset{[n]\choose k}$ and that $\mathcal J_i$ is intersecting for every $i\in [k]$. Moreover, $\gamma(\mathcal J_i) = 2$ for $i>1$ and $\mathcal J_1$ is the Hilton--Milner family. It is an easy calculation to see that $|\mathcal J_i|>|\mathcal J_{i+1}|$ for every $k\ge 4$ and $i\in [k-1]$.\footnote{Indeed, the difference $|\mathcal J_i|-|\mathcal J_{i+1}|$ is ${n-k-2\choose k-2}-1$ for $i=1$ and ${n-k-i\choose k-1}-{n-k-i-1\choose k-1}={n-k-i-1\choose k-2}$ for $i\ge 2$.}

\begin{thm}[\cite{HK}]\label{thmhk} Let $n>2k$, $k\ge 4$. Then any intersecting family $\ff$ with $\gamma(\ff)\ge 2$ satisfies
\begin{equation}\label{eqhk}|\ff|\le {n-1\choose k-1}-{n-k-1\choose k-1}-{n-k-2\choose k-2}+2,\end{equation}
moreover, for $k\ge 5$ the equality is attained only on the families isomorphic to $\mathcal J_2$.
\end{thm}
We note that Han and Kohayakawa also proved their theorem for $k= 3$, as well as described the cases of equality for  $k=4$. These cases are more tedious and do  not follow from our methods in a straightforward way. However, Theorem~\ref{thmhk} can be deduced without much effort. A slightly weaker version of the theorem above (without uniqueness) is a consequence of the main result in the paper by Hilton and Milner \cite{HM} (cf. also \cite{HK}).

Applying Corollary~\ref{corhm} with $u=k$, we conclude that \eqref{eqhk} holds for $k\ge 4$. The bound is sharp, as witnessed by $\mathcal J_2$. The uniqueness requires hardly more effort, but since it uses Theorem~\ref{thmfull1}, stated in Section~\ref{sec3}, we postpone its proof until Section~\ref{sec4}.

For any set $X$, family $\ff\subset 2^X$ and $i\in X$, we use the following standard notations
\begin{align*}
\ff(\bar i):=&\ \{F\ :\ i\notin F\in\ff\}\ \ \ \text{and}\\
\ff(i):=&\ \{F\setminus\{i\}\ :\ i\in F\in\ff\}.
\end{align*}
Note that $\ff(\bar i)\subset {X\setminus \{i\}\choose k}$ and $\ff(i)\subset {X\setminus \{i\}\choose k-1}$.

Denote by $\mathcal E_{l}$ the maximal intersecting family with $|\mathcal E_l(\bar 1)|=l$, $|\bigcap_{E\in \mathcal E_l(\bar 1)}E|=k-1$ (note that the family is defined up to isomorphism). Note that $\mathcal J_2$ is isomorphic to $\mathcal E_2$.\footnote{Cf. the discussion in the beginning of Section~\ref{sec31} and note the relation to the lexicographic families, defined in Section~\ref{sec3}: $\mathcal E_l(\bar 1)$ is isomorphic to $\mathcal L([2,n],l,k)$.
Using Theorems~\ref{thmfull2} and~\ref{thmfulleq} (or by tedious direct calculation), we can conclude that $|\mathcal E_{k-1}|<|\mathcal E_{k-2}| = |\mathcal E_{n-k}|<|\mathcal E_{k-3}|<\ldots<|\mathcal E_{1}|$.
}
 It is not difficult to see that for $k-1<l<n-k$ we have $\mathcal E_l\subset \mathcal E_{n-k}$: we have $\mathcal E_l(1) = \mathcal E_{n-k}(1)$ for this range.
 The following theorem is one of the main results in \cite{KostM}:
\begin{thm}[\cite{KostM}]\label{thmko} Let $k\ge 5$ and $n=n(k)$ be sufficiently large. If $\ff\subset {[n]\choose k}$ is intersecting and $|\ff|> |\mathcal J_3|$ then  $\ff\subset \mathcal E_l$ for $l\in \{0,\ldots k-1, n-k\}$.
\end{thm}
We note that it is easy to verify that $|\mathcal J_3|<|\mathcal E_{n-k}|$, e.g., for $n>4k$.
The authors of \cite{KostM} we using the delta-systems method of Frankl.\footnote{The goal of their paper, was, in a sense, to draw the attention of the researchers to this method.} Actually, Theorem~\ref{thmko} can be deduced from the results of Frankl \cite{F16} with little extra effort.

Many results in extremal set theory are much easier to obtain once one assumes that $n$ is sufficiently large in comparison to $k$. (The possibility to apply the delta-method is one of the reasons.) In particular, the bound on $n$ in Theorem~\ref{thmko} is doubly exponential in $k$.
In this paper, we deduce Theorem~\ref{thmko} from a much more general result, which, additionally, holds without any restriction on $n$.

 A family is called \underline{minimal} with respect to some property, if none of its proper subfamilies possesses the property. For shorthand, we say that $\mm$ is \underline{minimal w.r.t. common intersection} if, for any $M_l\in \mm$, we have $|\bigcap_{M\in \mm\setminus\{M_l\}}M|>|\bigcap_{M\in \mm}M|$.
The following theorem is one of the main results of this paper.

\begin{thm}\label{thmclass2} Assume that $n>2k\ge 8$. Consider an intersecting family $\ff\subset{[n]\choose k}$ with  $\Delta(\ff) = d_1(\ff)$.
Take a  subfamily $\mathcal M\subset \ff(\bar 1)$, which is minimal w.r.t. common intersection and such that $|\bigcap_{M\in \mm}M|= t$. Take the (unique) maximal intersecting family $\ff'$, such that $\ff'(\bar 1) = \mm$. If $t\ge 3$
then we have
\begin{equation}\label{eqclass1} |\ff|\le |\ff'|,
\end{equation}
and, for $k\ge 5$ equality is possible if and only if $\ff$ is isomorphic to $\ff'$.

Moreover, if $\ff$ is as above and  $|\bigcap_{F\in \ff(\bar 1)}F|\le t$ for some $t\ge 3$, then
\begin{equation}\label{eqclass2} |\ff|\le |\mathcal J_{k-t+1}|,
\end{equation}
and, for $k\ge 5$, equality is possible only if $\ff$ is isomorphic to $\mathcal J_{k-t+1}$.
\end{thm}
This theorem generalizes Theorems~\ref{thmhk} and~\ref{thmko} and gives a reasonable classification of {\it all} large intersecting families.
We also note that we cannot in general replace the condition on $t$
by $t\ge 2$.
Indeed, one can see that the family $\mathcal H_2$ (cf. \eqref{eqhu}) is much larger than $\mathcal J_{k-1}$ for large $n$.\footnote{The size of $\mathcal J_{k-1}$ can be bounded by ${n-2\choose k-2}+{n-3\choose k-2}+2+(k-2)\big({n-4\choose k-2}-{n-k-2\choose k-2}\big)<{n-2\choose k-2}+2{n-3\choose k-2}$ for $n>2k^2$, say.}

\begin{cor}\label{cormk} The statement of Theorem~\ref{thmko} is valid for any $n>2k\ge 10.$
\end{cor}

\begin{proof}[Proof of Corollary~\ref{cormk}]
Fix any $\ff$ as in Theorem~\ref{cormk} and assume that $\Delta(\ff) = d_1(\ff)$. First assume that $|\bigcap_{F\in\ff(\bar 1)}|\le k-2$. Since $k\ge 5$, we are in position to apply the second part of Theorem~\ref{thmclass2} to $\ff$, and get a contradiction with $|\ff|>|\mathcal J_3|$.
Therefore, $|\bigcap_{F\in \ff(\bar 1)}F|=k-1$ and thus $\ff(\bar 1)$ is isomorphic to $\mathcal E_l(\bar 1)$ for $l = |\mathcal F(\bar 1)|$, which concludes the proof.
\end{proof}

The proof of Theorem~\ref{thmclass2} is given in Section~\ref{sec41}. The main tool of the proof of Theorem~\ref{thmclass2} is again the aforementioned bipartite switching trick. In this paper, we exploit it to much greater extent than in  \cite{FK1} and \cite{KZ}. One key observation that allows us to prove Theorem~\ref{thmclass2} is that this bipartite switching is possible even in situations when we know practically nothing about the structure of the family.\\

In the remaining part of the introduction, we present two results that extend some classic results of Frankl, Erd\H os Rothshild and Semeredi, and Furedi  proven for $n>n_0(k)$ (with double-exponential dependency on $k$, coming from the aforementioned delta-method) to the range $n>Ck$, where $C$ is an absolute constant. Apart from the bipartite switching trick, we use as a main tool the junta approximation theorem due to Dinur and Friedgut \cite{DF}.  The approach is resemblant of the recent paper \cite{Kup21} of the author, but here we apply it to a much wider class of problems. The rough framework of combining junta method to get approximate structure with combinatorial arguments for finer structure was used in an excellent recent paper \cite{KL}. One novel aspect in the use of junta method in the first problem below is that the actual extremal configuration is quite far from being a junta (i.e., the family is not defined by the intersection with a constant-size subset of the ground set, see precise definition in Section~\ref{sec5})! This, of course, poses additional complications.

For a family $\ff$, let $\tau(\ff)$ denote the \underline{covering number} of $\ff$, that is, the minimum size of set $S$ that intersects all sets in $\ff$. Each such $S$ we call a \underline{hitting set}.

Intersecting families of $k$-sets with fixed covering number were studied in several classical works. The Erd\H os--Ko--Rado theorem shows that the largest intersecting family of $k$-element sets has covering number $1$. The aforementioned result of Hilton and Milner \cite{HM} determined the largest intersecting family with covering number $2$. It is clear that any $k$-uniform intersecting family $\ff$ satisfies $\tau(\ff)\le k$: indeed, any set of $\ff$ is a hitting set for $\ff$. In a seminal paper \cite{EL}, Erd\H os and Lov\'asz proved that an intersecting family $\ff\subset {[n]\choose k}$ with $\tau(\ff)=k$ has size at most $k^k$ (note that it is independent of $n$!) and provided a lower bound of size roughly $(k/e)^k$. Later, both lower \cite{FOT2} and upper \cite{Che, AR, Fra21} bounds were improved.

Let us define the following important family.
\begin{equation}\label{deft2}
\mathcal T_2(k):=\big\{[k]\big\}\cup \big\{\{1\}\cup [k+1,2k-1]\big\}\cup \big\{\{2\}\cup [k+1,2k-1]\big\}.\end{equation}
It is easy to see that $\mathcal T_2(k)$ is intersecting, moreover, $\tau(\mathcal T_2(k))=2$.

In \cite{F16}, Frankl studied the following question: what is the size $c(n,k,t)$ of the largest intersecting family $\ff\subset {[n]\choose k}$ with $\tau(\ff)\ge t$? Define $\mathcal C_3(n,k)\subset {[n]\choose k}$ to be the maximal intersecting family with $\mathcal C_3(n,k)(\bar 1)$ isomorphic to $\mathcal T_2(k)$. It is easy to see that $\tau(\mathcal C_3(n,k))=3$. Frankl managed to prove the following theorem.

\begin{thm}[\cite{F16}]\label{thmtau3} Let $k\ge 3$ and $n\ge n_0(k)$. Then $c(n,k,3)= |\mathcal C_3(n,k)|$. Moreover, for $k\ge 4$ the equality holds only for families isomorphic to $\mathcal C_3(n,k)$.
\end{thm}

As in the case of Theorem~\ref{thmko}, the bound on $n$ is doubly exponential in $k$. One of the main results of this paper is the proof of Theorem~\ref{thmtau3} under much milder restrictions.

\begin{thm}[\cite{F16}]\label{thmtau3} The conclusion of Theorem~\ref{thmtau3} holds for any $n>Ck$, where $C$ is an absolute constant, independent of $k$.
\end{thm}
An important tool in the proof (as well as in the proof of Theorem~\ref{thmfull2}) is Lemma~\ref{lemmin}, in which we found an elegant way to bound the sizes of families $\ff$ satisfying $\tau(\ff)=3$ and  $\tau(\ff(\bar 1))=2$, and where $\ff(\bar 1)$ is minimal w.r.t. this property.

In the paper \cite{FOT1}, the authors managed to extend the result of \cite{F16} to the case $\tau =4$, determining the exact value of $c(n,k,4)$ and the structure of the extremal family for $n>n_0(k)$. The analysis in \cite{FOT1} is much more complicated than that in \cite{F16}, and the problem for $\tau\ge 5$ is still wide open. The result of \cite{FOT1} may be extended to much smaller $n$ in a similar way, but, of course, progress on the case $\tau \ge 5$ would be more interesting.\\

Erd\H os, Rothschild and Szemer\'edi (cf. \cite{Erd73}) raised the following question: how large can the intersecting family $\ff\subset {[n]\choose k}$ be, given that $\Delta(\ff)\le c|\ff|$. They proved that, for any fixed $2/3<c<1$, there exists $n_0(k,c)$, such that for any $n>n_0(k,c)$ the largest intersecting family under this restriction is $\mathcal H_2$ (see \eqref{eqhu}), up to isomorphism. Stronger results of this type were proven by Frankl \cite{F16}  and by F\"uredi \cite{Fur}. See also a survey \cite{DeF}, where the concise statement of the results is given.

The methods we developed in this paper allow us to extend these results to the range $n>Ck$, where $C$ depends on $c$ only. We are going to illustrate it for one such theorem, but we note that the same ideas would work for other cases.
Let us state one of the theorems, proven in \cite{F16}. Recall that a \underline{Fano plane} (projective plane of order $2$) $\mathcal P$ is a family consisting of $7$ $3$-element sets $P_1,\ldots, P_7$, such that $|P_i\cap P_j|=1$ for each $i\ne j$.
\begin{thm}[\cite{F16}]\label{thmfano}  Suppose that $\ff\subset {[n]\choose k}$ is intersecting, and $\Delta(\ff)\le c|\ff|$ for some $c\in (3/7,1/2)$. Assume that $\ff$ has the largest cardinality among such families. Then, for any $n>n_0(c,k)$, $\ff$ is isomorphic to
$$\mathcal D_{3/7}:=\ \Big\{F\in {[n]\choose k}\ :\  P\subset F\text{ for some }P\in \mathcal P\Big\}.$$
\end{thm}

Our contribution here is as follows.
\begin{thm}\label{thmbounddeg} There exists an absolute constant $C$, such that the following holds. In terms of Theorem~\ref{thmfano}, if there exists $\epsilon>0$, such that $c\in [3/7+\epsilon, 1/2-\epsilon]$, then the conclusion of Theorem~\ref{thmfano} holds for $n>\frac{C}{\epsilon}k$.
\end{thm}
The proof of this theorem is given in Section~\ref{sec52}. In Section~\ref{sec6}, we discuss our results and methods, as well as pose some open problems.

\section{The complete diversity version of Frankl's theorem}\label{sec3}
 Let us begin with the statement of the Kruskal--Katona theorem. We shall state the Kruskal--Katona theorem in two different forms, but the most handy for our purposes is the form due to Hilton in terms of cross-intersecting families. Let us first give some definitions.

For a set $X$, \underline{lexicographical order} (lex) $\prec$ on the sets from ${X\choose k}$ is a total order, in which $A\prec B$ iff  the minimal element of $A\setminus B$ is smaller than the minimal element of $B\setminus A$.
 For $0\le m\le {|X|\choose k}$ let $\mathcal L(X,m,k)$ be the collection of $m$ largest $k$-element subsets of $X$ with respect to lex.
 We say that two families $\mathcal A,\mathcal B$ are \underline{cross-intersecting} if $A\cap B\ne \emptyset$ for any $A\in\mathcal A, B\in \mathcal B$.

\begin{thm}[\cite{Kr},\cite{Ka}]\label{thmHil}If $\mathcal A\subset {[n]\choose a}, \mathcal B\subset {[n]\choose b}$ are cross-intersecting then the families $\mathcal L([n],|\mathcal A|,a),\mathcal L([n],|\mathcal B|, b)$ are also cross-intersecting.
\end{thm}

In this section, we analyze in great details the relationship between the diversity of an intersecting family and its size. We first note that, if the value of diversity is given precisely, then it is  easy to determine the largest intersecting family with such diversity. Indeed, the subfamilies of sets containing the element of the largest degree and not containing the element of the largest degree are cross-intersecting, and one can get exact bounds using Theorem~\ref{thmHil}. Studying the size of an intersecting family with given upper bounds on diversity is not interesting: the largest intersecting family has diversity 0.

In this section we obtain the concluding version of Theorem~\ref{thm1}, which tells {\it exactly}, how large an intersecting family may be, given a {\it lower} bound on its diversity. We determine all ``extremal'' values of diversity and the sizes of the corresponding families.

The difficulty to obtain such a version of Theorem~\ref{thm1} is that, while Theorem~\ref{thmHil} gives a very strong and clear characterisation of families with fixed diversity, the size of the family is not monotone w.r.t. diversity (the size of the largest family with given diversity does not necessarily decrease as diversity increases, although it is true in ``most'' cases). Moreover, the numerical dependence between maximal possible sizes of $\aaa$ and $\bb$ given by  Theorem~\ref{thmHil} is complicated and difficult to work with, see \cite{FMRT} and \cite{FT7}. Thus, an effort is needed to find the right point of view on the problem.\vskip+0.1cm

We shall give two versions of the main theorem of this section. First, we give a ``quantitative'' version with explicit sharp bounds on the size of intersecting families depending on the lower bound on their diversity. It may be more practical to apply in some cases, but it is difficult to grasp what is hidden behind the binomial coefficients in the formulation. Thus, later in the section (and as an intermediate step of the proof), we shall give a ``conceptual'' version of our main theorem. We note that the proof that we present is completely computation-free.  In Sections~\ref{sec33},~\ref{sec34} we present strengthenings and generalisations of our main result. The cases of equality in Theorems~\ref{thmfull1} and~\ref{thmfull2} are described in Theorem~\ref{thmfulleq}. We provide a weighted version of the main result and a generalization to the case of cross-intersecting families.

We note that the main results of the section are meaningful for any $k\ge 3$. (This is by no means a serious restriction since possible structure of intersecting families in ${[n]\choose k}$ for $k\le 2$ is trivial.)\vskip+0.1cm

The following representation of natural numbers is important for the (classic form of) the Kruskal--Katona theorem. Given positive integers $\gamma$ and $k$, one can always write down $\gamma$ {\it uniquely} in the \underline{$k$-cascade form}:
$$\gamma = {a_k\choose k}+{a_{k-1}\choose k-1}+\ldots +{a_s\choose s}, \ \ a_k>a_{k-1}>\ldots >a_s\ge 1.$$

For the sake of comparison, let us state the classical version of the Kruskal--Katona theorem (equivalent to Theorem~\ref{thmHil}).

\begin{thm}[\cite{Kr},\cite{Ka}]\label{thmkk} Let $\ff\subset {[n]\choose k}$ and put $$\partial(\ff):=\big\{F'\in {[n]\choose k-1}\ :\  F'\subset F \text{ for some }F\in \ff\big\}.$$ If $|\ff| = {a_k\choose k}+\ldots +{a_s\choose s}$, then
$$|\partial(\ff)|\ge {a_k\choose k-1}+{a_{k-1}\choose k-1}+\ldots +{a_s\choose s-1}.$$
\end{thm}

 Next, we start preparations to state the main result of this section. Given a number $\gamma<{n-1\choose k-1}$, let us write it in the  $(n-k-1)$-cascade form:
  $$\gamma = {n-b_1\choose n-k-1}+{n-b_{2}\choose n-k-2}+\ldots+{n-b_{s_b}\choose n-k-s_b},$$
  where $1< b_1<b_2<\ldots<b_{s_b}.$\footnote{Note that if $b_1=1$ then $\gamma\ge {n-1\choose k}\ge {n-1\choose k-1}$, which contradicts our assumption on $\gamma$.}
   Define $T_{\gamma}:=\{b_1,\ldots,b_{s_b}\}$ and put $S_{\gamma}:=T_{\gamma}\,\oplus\, [2,b_{s_b}-1],$ where $\oplus$ stands for symmetric difference. Note that $S_{\gamma}\cup T_{\gamma} = [2,b_{s_b}]$ and $S_{\gamma}\cap T_{\gamma} = \{b_{s_b}\}$. Suppose that $S_{\gamma} = \{a_1,\ldots, a_{s_a}\}$, where $1< a_1<\ldots<a_{s_a} := b_{s_b}$.

\begin{defi}\label{defresnu}   We call a nonnegative integer $\gamma$ \underline{resistant}, if either $\gamma={n-4\choose k-3}$ or the following holds:
   \begin{enumerate}
   \item $s_a := |S_{\gamma}|\le k$ and $s_b := |T_{\gamma}|\le k-1$;
   \item $b_i>2i+2$ for each $i\in [s_b]$.
   \end{enumerate}
\end{defi}

In particular, any integer $\gamma>{n-4\choose k-3}$ has ${n-4\choose k-3}$ as one of the summands in the $(n-k-1)$-cascade form, and thus is not resistant.

Let $0= \gamma_0<\gamma_1<\ldots <\gamma_{m} = {n-4\choose k-3}$ be all the resistant numbers in increasing order. 

\begin{thm}\label{thmfull1}  Let $n>2k\ge6$. Consider an intersecting family $\ff\subset {[n]\choose k}$. Suppose that $\gamma_{l-1}<\gamma (\ff)\le\gamma_l$ for $l\in [m]$ and that the representation of $\gamma_l$ in the $(n-k-1)$-cascade form is
$$\gamma_l = {n-b_1\choose n-k-1}+{n-b_{2}\choose n-k-2}+\ldots +{n-b_{s_b}\choose n-k-s_b},$$
then \begin{equation}\label{eqfull1}|\ff|\le {n-a_1\choose n-k}+{n-a_{2}\choose n-k-1}+\ldots +{n-a_{s_a}\choose n-k-s_a+1}+\gamma_l,\end{equation}
where $\{b_1,\ldots, b_{s_b}\} = T_{\gamma_l}$ and $\{a_1,\ldots, a_{s_a}\} = S_{\gamma_l}$.

 The expression in the right hand side of \eqref{eqfull1} strictly decreases as $l$ increases.

Moreover, the presented bound is sharp: for each $l =1,\ldots, m$ there exists an intersecting family with diversity $\gamma_l$ which achieves the bound in \eqref{eqfull1}.
\end{thm}

Theorem~\ref{thm1} and the concept of diversity was successfully used to advance in several problems concerning intersecting families \cite{Feg, FK5, IK, Kup22}. However, sometimes the statement of Theorem~\ref{thm1} was not fine-grained enough for the applications, and some extra work was needed to deduce some rudimentary versions of Theorem~\ref{thmfull1}. (One example of an application of an easy application of Theorem~\ref{thmfull1} is Theorem~\ref{thmhk} in the introduction.) This is one of our motivations for proving Theorem~\ref{thmfull1}. The other motivation is the desire to have the conclusive version of Theorems~\ref{thmfr} and~\ref{thm1} and highlight the parallel with the original version of the Kruskal--Katona theorem.

Let us mention that we state Theorem~\ref{thmfull1} only for $\gamma(\ff)\le {n-4\choose k-3}$, since Theorem~\ref{thm1} already gives us the bound $|\ff|\le {n-2\choose k-2}+2{n-3\choose k-2}$ if $\gamma(\ff)\ge {n-4\choose k-3}$, and we cannot get any better bound in general for larger $\gamma$. Indeed, the intersecting family $\mathcal H_2$ (cf. \eqref{eqhu}) attains the bound above on the cardinality and has diversity ${n-3\choose k-2}$. The second part of Corollary~\ref{corhm} complements Theorem~\ref{thm1} in this respect, showing the aforementioned example is essentially the only exception. We also note that in a recent work \cite{Kup21} the author managed to prove that there exists $c$, such that for any $n,k$ satisfying $n\ge ck$ we have $\gamma(\ff)\le {n-3\choose k-2}$ for any intersecting  $\ff\subset{[n]\choose k}$.
Earlier, this was shown to be true for $n\ge 6k^2$ by Frankl \cite{Fra6}. Thus, for $n\ge ck$ Theorem~\ref{thmfull1} gives a complete answer to the question we address.

In Section~\ref{sec33}, we shall analyze the cases when the equality in \eqref{eqfull1} can be attained, as well as the cases of very large diversity. In particular, we shall prove the $u=3$ part of Corollary~\ref{corhm}.

\subsection{Some examples and Theorem~\ref{thmfull1} restated}\label{sec31}
Let us try to familiarize the reader with the statement of Theorem~\ref{thmfull1}. We have $\gamma_i = i$ for $i = 1,\ldots, k-3$. Indeed,  for $1\le \gamma <n-k-1$ we have $\gamma  = {n-k-1\choose n-k-1}+{n-k-2\choose n-k-2}+\ldots +{n-k-\gamma\choose n-k-\gamma}$. Thus, for any such $\gamma$ we have $T_{\gamma} = [k+1,k+\gamma]$ and $S_{\gamma} = [2,k]\cup \{\gamma\}$. Condition (1) in Definition~\ref{defresnu}  is satisfied if $\gamma\le k-1$. Condition (2) is satisfied iff $k+\gamma>2\gamma+2$, which is equivalent to $\gamma\le k-3$. (Note also that $\gamma=1$ is resistant for $k=3$.) From the discussion above, we also conclude that, for $k>3$, we have $\gamma_{k-2} = {n-k\choose n-k-1}=n-k.$

The following observation is given  for the sake of familiarizing the reader with the statement.
\begin{obs} The bound in Theorem~\ref{thmfull1} is always at least as strong as the bound in Theorem~\ref{thm1} for intersecting $\ff\subset {[n]\choose k}$ with diversity $\gamma(\ff)\le {n-4\choose k-3}$.
\end{obs}
Let us first compare the statement of Theorem~\ref{thmfull1} with the statement of Theorem~\ref{thm1} for $\gamma_l := {n-u-1\choose n-k-1}$ with integer $u$. Such $\gamma_l$ is resistant for any $u\in [3,k]$, and we have $A_{\gamma_l}=[2,u+1]$.
Thus, if we substitute such $\gamma_l$ in \eqref{eqfull1}, then we get the bound $$|\ff|\le {n-2\choose n-k}+\ldots +{n-u-1\choose n-k-u+1}+\gamma_l = {n-1\choose n-k}-{n-u-1\choose n-k-u}+{n-u-1\choose n-k-1},$$ which is exactly the bound \eqref{eq01} for $\gamma_l$.

Remark that the inequality \eqref{eqfull1}, however, gives a stronger conclusion and is valid in weaker assumptions. Indeed, while we know that this bound is sharp for $\gamma(\ff) = \gamma_l$, Theorem~\ref{thmfull1} also tells us that $\ff$ with $\gamma(\ff)>\gamma_l$ has strictly smaller size (quantified in Corollary~\ref{corhm}). Moreover, even if $\gamma_{l-1}<\gamma(\ff)\le \gamma_l$, we are still getting the same upper bound.

Returning to the proof of the observation, assume now that $u$ is a real number. The function in the right hand side of \eqref{eq01} is  monotone decreasing as $\gamma(\ff)$ increases (or, equivalently, as $u$ decreases). Therefore, to show that the bound \eqref{eqfull1} is stronger than \eqref{eq01}, it is sufficient to verify it for all values of $\gamma_l$, $l=0,\ldots, m$. But for each of these values the bound \eqref{eqfull1} is sharp, so \eqref{eq01} can be only weaker than \eqref{eqfull1}.

\begin{proof}[Proof of the first part of Corollary~\ref{corhm}] Here, we prove  Corollary~\ref{corhm} for (integer) $u\ge 4$. Choose $l$ such that $\gamma_l={n-u-1\choose n-k-1}$. Assume first that $k\ge 5$. Then $\gamma_{l+1}={n-u-1\choose n-k-1}+{n-k-2\choose n-k-2}={n-u-1\choose n-k-1}+1=:\gamma$. To see this, we just have to check that $\gamma$ is a resistant number. In terms of Definition~\ref{defresnu}, we have $T_{\gamma}=\{u+1,k+2\}$ and $S_{\gamma}=\{2,3,\ldots, u,u+2\ldots, k+2\}.$ We have $s_b=2\le k-1$ and $s_a=k$, thus, condition~(1) of Definition~\ref{defresnu} is fulfilled. Moreover, $b_1=u+1>2\cdot 1+2$ and $b_2=k+2>2\cdot 2+2$, fulfilling condition~(2) of the definition. Thus, \eqref{eqfull1} implies that
\begin{footnotesize}\begin{align*}|\ff|\le& {n-2\choose n-k}+\ldots+{n-u\choose n-k-u+2}+{n-u-2\choose n-k-u+1}+\ldots+{n-k-2\choose n-2k+1}+{n-u-1\choose n-k-1}+1\\
=&{n-1\choose n-k}-{n-u-1\choose n-u-k}+{n-u-1\choose n-k-1}+1-\epsilon,\end{align*}\end{footnotesize}
where
\begin{footnotesize}$$\epsilon={n-u-1\choose n-k-u+1}-\Big({n-u-2\choose n-k-u+1}+\ldots+{n-k-2\choose n-2k+1}\Big)={n-k-2\choose n-2k}={n-k-2\choose k-2}.$$\end{footnotesize}
If $k=4$ then $u=4$ and $\gamma_{l+1}={n-4\choose k-3}={n-4\choose 1}$. We get that
\begin{footnotesize}$$|\ff|\le {n-1\choose k-1}-{n-4\choose k-1}+{n-4\choose k-3}={n-1\choose k-1}-{n-5\choose k-1}+{n-5\choose k-4}-\epsilon,$$\end{footnotesize} where
\begin{footnotesize}$$\epsilon = {n-5\choose k-2}-{n-5\choose k-3}={n-5\choose 2}-{n-5\choose 1}={n-6\choose 2}-1.$$
\end{footnotesize}
The second part of the corollary is proved in Section~\ref{sec33}.
\end{proof}


Our next goal is to state the ``conceptual'' version of Theorem~\ref{thmfull1}.
It requires certain preparations. We will use the framework and some of the ideas from \cite{FK1}, as well as from \cite{KZ}. First of all, we switch to the cross-intersecting setting. Given an intersecting family $\ff$ with $\Delta(\ff)=\delta_1(\ff)$, consider the families
\begin{align*}
\ff(1):=&\{F\setminus\{1\}\ :\ 1\in F\in\ff\}\ \ \ \ \ \text{and}\\
\ff(\bar 1):=&\{F\ :\ 1\notin F\in\ff\}.
\end{align*}
Remark that $\gamma(\ff)=|\ff(\bar 1)|$. Applying Theorem~\ref{thmHil}, from now on and until the end of Section~\ref{sec3} we assume that $\ff(1)=\mathcal L([2,n],|\ff(1)|,k-1)$ and $\ff(\bar 1) = \mathcal L([2,n],|\ff(\bar 1)|, k)$. Note that $\ff(1),\ff(\bar 1)\subset 2^{[2,n]}$. For shorthand, we denote $\aaa:=\ff(1), \bb:=\ff(\bar 1)$.  While proving Theorem~\ref{thmfull1}, we
will work with the ground set $[2,n]$, in order not to confuse the reader and to keep clear the relationship between the diversity of intersecting families and the sizes of pairs of cross-intersecting families.\vskip+0.1cm

Both $\aaa$ and $\bb$ are determined by their lexicographically last set. In this section, we use the lexicographical/containment order on $2^{[2,n]}$, which is defined as follows: $A\prec B$ iff $A\supset B$ or the minimal element of $A\setminus B$ is smaller than the minimal element of $B\setminus A$. Let us recall some notions and results from \cite{FK1} related to the Kruskal--Katona theorem and cross-intersecting families. For sets $S$ and $X$,  $|S\cap X|\le a$, we define
$$\mathcal L(X,S,a):=\Big\{A\in {X\choose a}\ :\ A\prec S\cap X\Big\}.$$
For example, the family $\{G\in{[2,n]\choose 10}\ :\  2\in G, G\cap\{3,4\}\ne \emptyset\}$ is the same as the family $\mathcal L([2,n],S,10)$ for $S=\{1,2,4\}$. If $\G=\mathcal L(X,S,a)$ for a certain set $S$, then we say that $S$ is the \underline{characteristic set} of $\G$.

Note that, in what follows, we shall work with $X=[2,n]$ and will omit $X$ from the notation for shorthand. For convenience, we shall assume that $1\in S$ (motivated by the fact that $S$ will stand for the characteristic set for the subfamily of all sets containing 1 in the original family), while $T\subset [2,n]$.

\begin{defi}\label{defrespa} We say that two sets $S\subset [n]$ and $T\subset [2,n]$  \underline{form a resistant pair}, if either $T =\{2,3,4\}$ and $S = \{1,4\}$,  or the following holds: assuming that the largest element of $T$ is $j$, we have
\begin{enumerate}
  \item $S\cap T = \{j\},\ S\cup T = [j]$, $|S|\le k,$ $|T|\le k$;
  \item for each $ i\ge 4$, we have $|[i]\cap S|< |[i]\setminus S|$.
\end{enumerate}
\end{defi}
Condition (2), roughly speaking, states that in $[i]$ there are more elements in $T$ than in $S$. Note that 2 implies that $T\supset \{2,3,4\}$ for each resistant pair. There is a close relationship between this notion and the notion of a resistant number, which we discuss a bit later. Let us first give the characteristic set version of Theorem~\ref{thmfull1}. For convenience, we put $T_0 = [2,n]$ to be the characteristic set of the empty family and $S_0:=\{1,n\}$ to be the characteristic set of the family ${[2,n]\choose a}$.

\begin{thm}\label{thmfull2} Let $n>2k\ge6$. Consider all resistant pairs $S_l\subset [n],\ T_l\subset [2,n]$, where $l\in [m]$. Assume that $T_0<T_1<T_2<\ldots <T_m$. Then
\begin{equation}\label{eqfull2} |\mathcal L(S_{l-1},k-1)|+|\mathcal L(T_{l-1},k)|>|\mathcal L(S_{l},k-1)|+|\mathcal L(T_{l},k)| \ \ \ \text{ for each }l \in [m],\end{equation}
and any cross-intersecting pair of families $\mathcal A\subset {[2,n]\choose k-1},\ \mathcal B\subset {[2,n]\choose k}$ with $|\mathcal L(T_{l-1},k)|<|\bb|\le |\mathcal L(T_l,k)|$ satisfies \begin{equation}\label{eqfull3} |\aaa|+|\bb|\le |\mathcal L(S_l,k-1)|+|\mathcal L(T_l, k)|.\end{equation}

In terms of intersecting families, if $\ff\subset {[n]\choose k}$ is intersecting and $|\mathcal L(T_{l-1},k)|<\gamma(\ff)\le |\mathcal L(T_l,k)|$, then $|\ff|\le |\mathcal L(S_l,k-1)|+|\mathcal L(T_l, k)|$.
\end{thm}

First, we remark that the intersecting part is clearly equivalent to the second statement of the cross-intersecting part. Second, Proposition~\ref{prop9} below shows that the families $L(S_l,k-1)$ and $L(T_{l},k)$ are cross-intersecting and thus \eqref{eqfull3} is sharp.

We say that two sets $S$ and $T$ in $[2,n]$ \underline{strongly intersect}, if there exists a positive integer $j$ such that $S\cap T\cap[2,j]=\{j\}$ and $S\cup T\supset [2,j]$.
 The following proposition was proven in \cite{FK1}:

\begin{prop}[\cite{FK1}]\label{prop9} Let $A$ and $B$ be subsets of $[2,n]$, $|A|\le a, |B|\le b$, and $|[2,n]|=n-1\ge a+b$. Then $\mathcal L(P,a)$ and $\mathcal L(Q,b)$ are cross-intersecting iff $P$ and $Q$ strongly intersect.
\end{prop}

Now let us deduce Theorem~\ref{thmfull1} from Theorem~\ref{thmfull2}. 

\begin{proof}[Reduction of Theorem~\ref{thmfull1} to Theorem~\ref{thmfull2}]
 Let us first compute the size of the family $\mathcal L(T, k-1)$ for a given set $T\in [2,n]$. Assume that $b_{s_b}$ is the largest element of $T$ and that $[2,b_{s_b}]\setminus T= \{b_1,\ldots, b_{s_b-1}\}$. Take the smallest element $b_1\in [2,n]\setminus T$ and consider the family with characteristic set $T_1:=(T\cap [b_1])\cup \{b_1\}$. Note that $T_1 = [2,b_1]$ and thus the size of this family is ${n-b_1\choose k-b_1+1} = {n-b_1\choose n-k-1}$. Since $T_1\prec T$, this family is a subfamily of  $\mathcal L(T, k-1)$. Proceeding iteratively, at each step take the next smallest (not chosen yet) element $b_i$ from $[2,n]\setminus T$ and define the set $T_i:=(T\cap [b_i])\cup b_i$. Again, $T_i\prec T$. Count the sets that belong to $\mathcal L(T_i,k-1)\setminus L(T_{i-1},k-1)=\big\{F\in {[2,n]\choose k}\ :\ F\cap [b_i]=T_i\big\}$. Their number is precisely ${n-b_i\choose k-|T_i|} = {n-b_i\choose n-k-|[b_i]\setminus T_i|}$. Since we ``stop'' at every element that is not included in $T$, we get that $|[b_i]\setminus T_i| = |[b_{i-1}\setminus T_{i-1}|+1 = \ldots = i$. Therefore,  ${n-b_i\choose n-k-|[b_i]\setminus T_i|}={n-b_i\choose n-k-i}$. We stop the procedure at the point when $T_i = T$, including the sets $F\in {[2,n]\choose k-1}$ that satisfy $F\cap [b_{s_b}] = T$ in the count. It should be clear that, in this counting procedure, we counted each set from $\mathcal L(T,k-1)$ exactly once. We get that
$$|\mathcal L(T,k-1)| = {n-b_1\choose n-k-1}+{n-b_2\choose n-k-2}+\ldots +{n-b_{s_b}\choose n-k-s_b},$$
and the displayed formula gives the representation of $|\mathcal L(T,k-1)|$ in the $(n-k-1)$-cascade form! Moreover, we conclude that the set $\{b_1,\ldots, b_{s_b}\}$ is exactly the set $T_{\gamma}$ for $\gamma:=|\mathcal L(T,k-1)|$ (cf. the paragraph above Definition~\ref{defresnu}). We have $T_{\gamma} = ([2,b_{s_b}]\setminus T)\cup \{b_{s_b}\}$ and thus $T = S_{\gamma}$.

Therefore, if $S,T$ is a resistant pair, then, putting $\gamma:=|\mathcal L(T,k-1)|$, we get that $T = S_{\gamma}$ and $S\cap [2,n] = T_{\gamma}$. This immediately implies that $T_{\gamma}$ and $S_{\gamma}$ satisfy condition (1) of Definition~\ref{defresnu}. The implication in the other direction follows as well.  Condition (2) of Definition~\ref{defrespa} is equivalent to the  statement that $1+|T_{\gamma}\cap [i]|<|[2,i]\setminus T_{\gamma}|$  for each $i\ge 4$, which, in turn, is equivalent to $b_l>2l+2$ for each $l\ge 1$. Finally, it is clear that $\gamma = {n-4\choose k-3}$ correspond to the characteristic set $\{2,3,4\}$.

We conclude that $T_l$ and $S_l$ form a resistant pair if and only if $|\mathcal L(T_l,k-1)|$ is a resistant number. Doing calculations as above, one can conclude that
$$|\mathcal L(S_l,k)| = {n-a_1\choose n-k}+{n-a_{2}\choose n-k-1}+\ldots +{n-a_{s_a}\choose n-k-s_a+1},$$
where $a_i$ are as in the statement of Theorem~\ref{thmfull1}.
Given that, it is clear that the inequality \eqref{eqfull2} is equivalent to the statement saying that the right hand side of \eqref{eqfull1} is strictly monotone, and that  \eqref{eqfull1} is equivalent to \eqref{eqfull3}.

Finally, the sharpness claimed in Theorem~\ref{thmfull1} immediately follows from the remark in the paragraph after Theorem~\ref{thmfull2}.
\end{proof}

\subsection{Proof of Theorem~\ref{thmfull2}} We say that $\mathcal A\subset {[2,n]\choose a}$ and $\mathcal B\subset {[2,n]\choose b}$ form a \underline{maximal} \underline{cross-intersecting pair}, if, whenever $\mathcal A'\subset {[2,n]\choose a}$ and $\mathcal B'\subset {[2,n]\choose b}$ are cross-intersecting with $\mathcal A'\supset \mathcal A$ and $\mathcal B'\supset \mathcal B$, then necessarily $\mathcal A = \mathcal A'$ and $\mathcal B = \mathcal B'$ holds.

The following proposition from \cite{FK1} is another important step in our analysis.

\begin{prop}[\cite{FK1}]\label{cross2} Let $a$ and $b$ be positive integers, $a+b\le n-1$. Let $P$ and $Q$ be non-empty subsets of $[2,n]$ with $|P|\le a$, $|Q| \le b$. Suppose that $P$ and $Q$ strongly intersect in their largest element, that is, there exists $j$ such that $P\cap Q = \{j\}$ and $P\cup Q = [2,j]$. Then $\mathcal L([2,n],P,a)$ and $\mathcal L([2,n],Q,b)$ form a maximal pair of cross-intersecting families.

Inversely, if $\mathcal L([2,n],m,a)$ and $\mathcal L([2,n],r,b)$ form a maximal pair of cross-intersecting families, then there exist sets $P$ and $Q$ that strongly intersect in their largest element, such that $\mathcal L([2,n],m,a)=\mathcal L([2,n],P,a)$, $\mathcal L([2,n],r,b)=\mathcal L([2,n],Q,b)$. \end{prop}

Recall that we aim to maximize $|\aaa|+|\bb|$ given a lower bound on $|\bb|$. The proof is based on the following two lemmas. 

\begin{lem}\label{lemdiv} Consider a pair of cross-intersecting families $\aaa\subset {[2,n]\choose k-1},\ \bb\subset {[2,n]\choose k}$. Suppose that $\aaa = \mathcal L(S,k-1),\ \bb=\mathcal L(T,k)$ for some sets $S\subset [n]$, $T\subset [2,n]$ that strongly intersect in their last element $j$. Suppose also that $T\precneqq \{2,3,4\}$.

Assume that $S$ and $T$ do not form a resistant pair, that is, there exists $5\le i\le j$, such that $\big|[i]\cap S\big|\ge \big|[i]\setminus S\big|$. Put $T':=[i]\setminus S$ and choose $S'$ so that it strongly intersects with $T'$ in its largest element. Then the families $\aaa'\subset {[2,n]\choose k-1},\ \bb'\subset {[2,n]\choose k}$ with characteristic sets $S', T'$ are cross-intersecting and  satisfy $|\aaa'|+|\bb'|\ge |\aaa|+|\bb|$ and $|\bb'|> |\bb|$.

Moreover, if $\big|[i]\cap S\big|>\big|[i]\setminus S\big|$ then  $|\aaa'|+|\bb'|> |\aaa|+|\bb|$.
\end{lem}

\begin{proof}[Proof of Lemma~\ref{lemdiv}] First, recall that $1\in S$. Since $S'$ and $T'$ are strongly intersecting, the families $\aaa'$, $\bb'$ are cross-intersecting. Next, clearly, $T'\subsetneq T$ and thus $\bb'\supsetneq \bb$. Therefore, we only have to prove that $|\aaa|+|\bb|\le |\aaa'|+|\bb'|$, and that the inequality is strict in the case indicated in the lemma.

Consider the following families:
\begin{align*}
\mathcal P_a:=&\big\{P\in {[2,n]\choose k-1}\ :\  P\cap [i] = [2,i]\cap S\big\},\\
\mathcal P_b:=&\big\{P\in {[2,n]\choose k}\ :\  P\cap [i] = [i]\setminus S\big\}.
\end{align*}
We have $\aaa\setminus \mathcal P_a = \aaa',$ $\bb\cup \mathcal P_b = \bb'$. Both equalities are proved in the same way, so let us show, e.g., the first one.  We have $S'\prec S\prec S\cap [i]$, therefore $\aaa'\subset \aaa\subset \mathcal L(S\cap [i],k-1)$. On the other hand, we claim that $S'$ and $S\cap [i]$ are two consecutive sets in the lexicographic order on $[i]$. Indeed, assume that the largest element of $T'$ is $j'$.
If $j' = i$, then $S'\supset S\cap [i]$, $(S'\setminus S)\cap [i] = \{i\}$, which proves it in this case. If $j'<i$, then $[j'+1,i]\subset S\cap [i]$, $j'\notin S\cap [i]$. It is easy to see that the set that precedes $S\cap [i]$ in the lexicographic order on $[i]$ ``replaces'' $[j'+1,i]$ with $\{j'\}$, that is, it is $S'$. Therefore,
$$\aaa \setminus \aaa' \subset \mathcal L(S\cap [i],k-1)\setminus \aaa' = \mathcal L(S\cap [i],k-1)\setminus \mathcal L(S',k-1) = \mathcal P_a,$$ which, together with the fact that $\mathcal P_a$ and $\aaa'$ are disjoint, is equivalent to the equality we aimed to prove.

Next, consider the bipartite graph $G$ with parts $\mathcal P_a,\mathcal P_b$ and edges connecting disjoint sets. Then, due to the fact that $\aaa$ and $\bb$ are cross-intersecting, $(\aaa\cap \mathcal P_a)\cup (\bb\cap \mathcal P_b)$ is an independent set in $G$.

The graph $G$ is biregular, and therefore the largest independent set in $G$ is one of its parts. We have $|\mathcal P_a| = {n-i\choose s_a}$, $|\mathcal P_b| = {n-i\choose s_b}$, where $s_a = k-|[i]\cap S|,\ s_b = k-|[i]\setminus S|$. By the condition from the lemma, we have $s_b\ge s_a$, and, since $n-i> s_a+s_b$, we have $|\mathcal P_b|\ge |\mathcal P_a|$. We conclude that $|\mathcal P_b|$ is the largest independent set in $G$, so $|\mathcal P_b|\ge (\aaa\cap \mathcal P_a)\cup (\bb\cap \mathcal P_b),$ and therefore
$$|\aaa'|+|\bb'|-(|\aaa|+|\bb|) = |\mathcal P_b|-|\aaa\cap \mathcal P_a|-|\bb\cap \mathcal P_b|\ge 0.$$
If  $\big|[i]\setminus S\big|< \big|S\cap [i]\big|$, then $|\mathcal P_b|> |\mathcal P_a|$ and $\mathcal P_b$ is the {\it unique} independent set of maximal size in $G$. Thus, we have strict inequality in the displayed inequality above.
\end{proof}

Slightly abusing notation, we say that $\aaa, \bb$ form a resistant pair if the corresponding characteristic sets form a resistant pair. The second lemma describes how do the resistant pairs behave.   More specifically, it shows that \eqref{eqfull2} holds: the sum of sizes of resistant cross-intersecting families increases as the size of the second family decreases.


\begin{lem}\label{lemres} Consider a resistant pair of cross-intersecting families $\aaa\subset{[2,n]\choose k-1},\ \bb\subset {[2,n]\choose k}$, with characteristic sets $S\subset [n],T\subset [2,n]$, respectively, and another such resistant pair $\aaa'\subset{[2,n]\choose k-1},\ \bb'\subset {[2,n]\choose k}$ with characteristic sets $S'\subset [n],T'\subset [2,n]$.

If $T'\precneqq T$, then $|\bb'|< |\bb|$ and $|\aaa|+|\bb|<|\aaa'|+|\bb'|$.
\end{lem}

Therefore, while for general lexicographic pairs of families the sum of sizes is not monotone w.r.t. the size of the second family, it is monotone for resistant pairs.

Remark that, since $T'\precneqq T$, we have $T'\ne \{2,3,4\}$ and thus $S',T'$ must satisfy Condition (2) of Definition~\ref{defrespa}. We also note that we do not use the property that $S,T$ form a resistant pair. The proof also works for $T' = T_0$ (recall that $T_0 =[2,n]$).

The proof of this lemma is based on biregular bipartite graphs and is very similar to the proof of Lemma~\ref{lemdiv}, although is a bit trickier.

\begin{proof}[Proof of Lemma~\ref{lemres}] First, it is clear that, in the conditions of the lemma, we have $|\bb'|< |\bb|$. The rest of the proof is concerned with the inequality on the sums of sizes. We  consider two cases depending on how do the sets $T$ and $T'$ relate.\\

\textbf{Case 1. \pmb{$T'\nsupseteq T$.} } Note that this condition, in particular, implies that $T\ne \{2,3,4\}$.  Find the smallest $i\ge 5$, such that exactly one of the sets $T,T'$ contain $i$. Since $T'\prec T$, we clearly have $i\in T',\ i\notin T$. Consider the set $T'' = T'\cap [i]$. Then we clearly have $T'\precneqq T''\precneqq T$ and $T''\subset T'$. Accordingly, put $S''$ to be $\{i\}\cup ([i]\setminus T'')$, and consider the cross-intersecting families $\aaa''\subset {[2,n]\choose k-1},\ \bb''\subset {[2,n]\choose k}$, which have characteristic vectors $S''$ and $T''$, respectively.
Note that $S''=S\cap [i]$ and thus the pair $\aaa'', \bb''$ is resistant.
We claim that $|\aaa''|+|\bb''|> |\aaa|+|\bb|.$

We prove the inequality above as in Lemma~\ref{lemdiv}, but the roles of $S$ and $T$ are now switched. Consider the bipartite graph $G$ with parts

\begin{align*}
\mathcal P_a:=&\big\{P\in {[2,n]\choose k-1}\ :\  P\cap [i] = [2,i]\setminus T\big\},\\
\mathcal P_b:=&\big\{P\in {[2,n]\choose k}\ :\  P\cap [i] = [i]\cap T\big\},
\end{align*}
and edges connecting disjoint sets.
Similarly, we have $\aaa\cup \mathcal P_a = \aaa'',$ $\bb\setminus \mathcal P_b = \bb''$. Indeed, let us verify, e.g., the second equality. All $k$-element sets $P$ such that $P\prec T''$ are in $\bb$ and in $\bb\setminus \mathcal P_b$, as well as in $\bb''$, since $T''\precneqq T\cap [i]$. On the other hand, if we restrict to $[i]$, the sets $T\cap [i]$ and $T''$ are consecutive in the lexicographic order, and so any set $B$ from $\bb$ such that $B\succneqq T''$ must satisfy $B\cap [i] = T\cap [i]$. Therefore, $\bb\setminus \bb''\subset \mathcal P_b$ and $\bb\setminus \mathcal P_b = \mathcal B''$.

The families $\aaa$ and $\bb$ are cross-intersecting, and so the set $(\aaa\cap \mathcal P_a)\cup (\bb\cap \mathcal P_b)$ is independent in $G$. On the other hand, the largest independent set in $G$ has size $\max\{|\mathcal P_a|, |\mathcal P_b|\}$. Since the pair $\aaa',$ $\bb'$ is resistant, we have that $|[i]\cap T|=|[i]\setminus S''| > |[i]\cap S''|=|[i]\setminus T|$, which implies $|\mathcal P_a|={n-i\choose k-|[i]\setminus T|}> {n-i\choose k-|[i]\cap T|} = |\mathcal P_b|$, and thus $\mathcal P_a$ is the (unique) largest independent set in $G$. We have

$$|\aaa''|+|\bb''|-(|\aaa|+|\bb|) = |\mathcal P_a|-|\aaa\cap \mathcal P_a|-|\bb\cap \mathcal P_b|>0,$$
and the desired inequality is proven. Therefore, when comparing $T'$ and $T$, we may replace $T$ with $T''$, or rather assume that $T\subset T'$. We have reduced Case 1 to the following case.\\

\textbf{Case 2. \pmb{$T'\supset T$.}\footnote{We note that, in this case, we do not use the fact that $\aaa,\bb$ form a resistant pair (and thus the proof works for $T=\{2,3,4\}$).} }  Assume that $t_1<t_2<\ldots <t_l$ are the elements forming the set $T'\setminus T$ and let $t_0$ be the largest element of $T$. Let us first show that, for each $i\in [l-1]$, we have  $$|[t_i]\setminus S'|\ge |[t_i]\cap S'|+2.$$ Indeed, find the largest element $t'<t_i,$ such that $t'\in S'$. We have $|[t_i]\setminus S'|>|[t']\setminus S'|>|[t']\cap S'|=|[t_i]\cap S'|$. For  $i=0,\ldots, l$, put $S_i:=[t_i]\cap S'\cup \{t_i\}$ and $T_i:=T'\cap [t_i]$. Observe that, for each $i\in[l]$, we have $|[t_i]\setminus S_i|\ge |[t_i]\cap S_i|$ by the displayed inequality. Moreover, the inequality is strict for $i=l$. For each $i=0,\ldots, l$, let $\aaa_i\subset {[2,n]\choose k-1},\bb_i\subset {[2,n]\choose k}$ be the pair of cross-intersecting families defined by characteristic sets $S_i,T_i$. Observe that $\aaa_0=\aaa,\bb_0=\bb$ and $\aaa_l=\aaa',\bb_l=\bb'$.

Finally, for each $i\in [l]$, we show that $|\aaa_i|+|\bb_i|\ge |\aaa_{i-1}|+|\bb_{i-1}|$, moreover, the inequality is strict for $i=l$. This is clearly sufficient to conclude the proof of the lemma. Consider the bipartite graph $G$ with parts
\begin{align*}
\mathcal P_a^i:=&\big\{P\in {[2,n]\choose k-1}\ :\  P\cap [i] = S_i\cap [2,i]\big\},\\
\mathcal P_b^i:=&\big\{P\in {[2,n]\choose k}\ :\  P\cap [i] = [i]\setminus S_i\big\},
\end{align*}
and edges connecting disjoint sets. As before, we have $\aaa_{i-1}\cup \mathcal P_a^i = \aaa_i,$ $\bb_{i-1}\setminus \mathcal P_b^i = \bb_i$.
Using the fact that $|[i]\setminus S_i|\ge|[i]\cap S_i|$ and that this inequality is strict for $i=l$, we conclude the proof as in the previous case.
\end{proof}

Now let us put the things together.

\begin{proof}[Proof of Theorem~\ref{thmfull2}] First, \eqref{eqfull2} follows from Lemma~\ref{lemres}. Next, given a pair of cross-intersecting families $\aaa,\bb$ as in the theorem, we may assume using Theorem~\ref{thmHil} and Proposition~\ref{cross2} that $\aaa, \bb$ are lexicographic families defined by characteristic sets $S,T$ that strongly intersect in their last coordinate. We may further assume that they do not form a resistant pair. Using Lemma~\ref{lemdiv} with the smallest $i$ satisfying its conditions,  replace $\aaa,\bb$ with the corresponding pair $\aaa', \bb'$ defined by characteristic sets $S',T'$. Remark that $T\precneqq T'$ and that, moreover, $\aaa',\bb'$ form a resistant pair by the choice of $i$. (Note here that if $i=5$ then the resulting characteristic sets are $T'=T_m=\{2,3,4\}$ and $S'=S_m=\{1,4\}$.) Therefore, if $T_{l-1}\precneqq T$  then $T_l\prec T'$, and therefore  $|\mathcal L(S_l,k-1)|+|\mathcal L(T_l,k)|\ge |\aaa'|+|\bb'|\ge |\aaa|+|\bb|$. This completes the proof of the theorem.
\end{proof}

\subsection{Equality in Theorems~\ref{thmfull1},~\ref{thmfull2} and families with large diversity}\label{sec33}
In the notations of the previous section, let $\aaa,\bb$ be cross-intersecting  and defined by the characteristic sets $S,T$, respectively, where $|S|,|T|\le k$. In this section, we determine, for which $T$, $T_{l-1}\prec T\prec T_l$, it is possible to have equality in \eqref{eqfull3}.
\begin{defi}\label{defneutral} We say that a pair of strongly intersecting sets $S,\ T$ as above is \underline{$T_l$-neutral}, if $T$ is obtained in the following recursive way:
\begin{enumerate}
\item $T_l$ is $T_l$-neutral;
\item If $T'$ is $T_l$-neutral, then the set $T:=T'\cup \{2|T'|\}$ is $T_l$-neutral.
\end{enumerate}
\end{defi}
In other words, to form all $T_l$-neutral sets, we start from the set $T_l$, add the element $2|T_l|$ and then continue adding every other element until we have $k$ elements in the set.

It is not difficult to see that any $T_l$-neutral pair $S, T$ actually satisfies $T_{l-1}\precneqq T\prec T_l$. Let us also note that, in terms of Definition~\ref{defneutral}, each newly formed $T_l$-neutral set is different from the previous one: indeed, from Definition~\ref{defrespa}, the largest element in $T_l$ is at most $2|T_l|-2$ (actually, it is at most $2|T_l|-3$ for all $l<m$ and equal to $2|T_l|-2$ in the case $l=m$), and every newly added element (by Condition (2) of Definition~\ref{defneutral}) is bigger by 2 than the previously added element.

\begin{thm}\label{thmfulleq} Let $n>2k\ge6$. Consider a pair $\aaa \subset {[2,n]\choose k-1},\ \bb\subset{[2,n]\choose k}$ defined by strongly intersecting sets $S, T$ that intersect in their largest element. If $T_{l-1}<T \le T_l$ for some $l\in[m]$, then equality in \eqref{eqfull3} holds if and only if the pair $S, T$ is $T_l$-neutral.
\end{thm}

\begin{proof} First, let us show that, for any $T_l$-neutral pair, we have equality in~\eqref{eqfull3}. We prove it inductively. It is clear for the pair with characteristic sets $S,T$, where $T=T_l$. Assuming it holds for $T'$, let us prove that it holds for $T:=T'\cup \{2|T'|\}$.

Put $x:=2|T|$ and consider the pairs of cross-intersecting families $\aaa, \bb$, $\aaa', \bb'$, corresponding to the $T_l$-neutral pairs of sets $S, T$ and $S', T'$, respectively.
By Definition~\ref{defneutral}, we have $$|[x]\cap S| = |[x]\setminus S|.$$
Therefore, applying the argument of Lemma~\ref{lemdiv} with
\begin{align*}
\mathcal P_a:=&\big\{P\in {[2,n]\choose k-1}\ :\  P\cap [x] = [2,x]\cap S\big\},\\
\mathcal P_b:=&\big\{P\in {[2,n]\choose k}\ :\  P\cap [x] = [x]\setminus S\big\},
\end{align*}
We get that $k-1-|[2,x]\cap S| = k-|[x]\setminus S|$, which implies $|\mathcal P_a| = |\mathcal P_b|$. Moreover, $\aaa\setminus \aaa' = \mathcal P_a,$ $\bb\setminus \bb' = \mathcal P_b,$ since the sets $S'$ and $S$ are both subsets of $[x]$ and are consecutive in the lexicographical order on $[x]$ (and the same holds for $T, T'$). Therefore, $|\aaa'|+|\bb'| = |\aaa|+|\bb|$.\\

In the other direction, take a set $T$, $T_{l-1}\precneqq T\precneqq T_l$, and its pair $S$. Consider the corresponding pair of cross-intersecting families $\aaa, \bb$ and assume that they satisfy equality in \eqref{eqfull3}. Then it is easy to see that $T\supset T_l$. (Otherwise, either $T\succ T_l$, or $T$ must contain and thus succeed some other resistant set, which succeeds$T_l$, again contradicting $T\precneqq T_l$.) Assuming that $x$ is the last element in $T$, we must have
$$|[i]\cap S|<|[i]\setminus S| \ \ \ \ \ \ \ \text{for }5\le i\le x-1.$$
Indeed, otherwise, considering the bipartite graph $G$ with parts $\mathcal P_a$, $\mathcal P_b$ as displayed above for that $i$, we would get that $|\mathcal P_a|\le |\mathcal P_b|$ and both $\mathcal P_a\cap \aaa$ and $\mathcal P_b\cap \bb$ are non-empty. In this case $|\mathcal P_a\cap \aaa|+|\mathcal P_b\cap \bb|<|\mathcal P_a|$, which means that the pair $\aaa', \bb'$ defined by the characteristic sets $T':=T\cap [i]$ and its pair $S'$ would satisfy $|\aaa'|+|\bb'|>|\aaa|+|\bb|$. Moreover, $T'\supset T_l$, so $T\precneqq T'\prec T_l$ and  $|\aaa'|+|\bb'|\le |\mathcal L(S_l,k-1)|+|\mathcal L(T_l,k)|$. This contradicts the equality for $\aaa,\bb$ in \eqref{eqfull3}.

Therefore, since the pair $S, T$ is not resistant, we have
$$|[x]\cap S|=|[x]\setminus S|.$$
(We cannot have ``$>$'', since otherwise we would have ``$\ge$'' for $i=x-1$.) Removing $x$, we get a set $T'$, and conclude that $x = 2|T'|$. By induction on the size of the set $T$, we may assume that $T'$ is $T_l$-neutral. But then $T$ is $T_l$-neutral as well.
\end{proof}

A slight modification of the argument above leads to the proof of the second part of Corollary~\ref{corhm}.
\begin{proof}[Proof of the second part of Corollary~\ref{corhm}]
In the framework of the previous proofs, replace $\ff(1)$, $\ff(\bar 1)$ with lexicographical families and put  $\aaa:=\ff(1)$, $\bb:=\ff(\bar 1)$. For $i\in [3]$, let $\aaa_i,\bb_i$ be the pairs of cross-intersecting families defined by characteristic sets $T^i,S^i$, where $T^1=T_m=\{2,3,4\}$, $T^2=\{2,3\}$ and $T^3=\{2,4\}$. It is easy to see that $|\bb_1|={n-4\choose k-3},\ |\bb_2|={n-3\choose k-2}$ and $|\bb_3|={n-3\choose k-2}+{n-4\choose k-2}$. 

If $\gamma(\ff)<{n-3\choose k-2}$ then consider the graph $G$ with parts $\mathcal P_a,\mathcal P_b$, where the parts contains the $k$-element sets that intersect $[4]$ in $\{1,4\}$ and in $\{2,3\}$, respectively. Then $|\mathcal P_a|=|\mathcal P_b|$ and $\aaa_1\setminus \aaa_2=\mathcal P_a$, $\bb_2\setminus \bb_1 = \mathcal P_b$. Moreover, $(\mathcal P_a\cap \aaa)\cup (\mathcal P_b\cap \bb)$ is independent in $G$. Thus, if $\aaa\notin \{\aaa_1,\aaa_2\}$ then the aforementioned independent set is strictly smaller than $\mathcal P_a,\mathcal P_b$, and thus $|\mathcal A|+|\mathcal B|<|\mathcal A_i|+|\mathcal B_i|$ for $i=1,2$.

Furthermore, provided that both $\mathcal P_a\cap \aaa$ and $\mathcal P_b\cap \bb$ are non-empty, then
$$|\mathcal P_a\cap \aaa|+|\mathcal P_b\cap \bb|\le {n-4\choose k-2}-{n-k-2\choose k-2}+1.$$
This is easy to deduce from the Kruskal--Katona theorem and was proven already by Hilton and Milner \cite{HM}. For a more general statement, see Corollary~\ref{corkz} in Section~\ref{sec34}.
Since $\gamma(\ff)>{n-4\choose k-3}$, $\mathcal P_b\cap \mathcal \bb$ is non-empty. If $\aaa$ is non-empty then the displayed bound holds, and thus \eqref{eqhm} is valid. If $\aaa$ is empty then, using $|\bb|\le {n-4\choose k-2}-{n-k-2\choose k-2}+1$ (that immediately follows from $|\ff(\bar 1)|\le {n-3\choose k-2}-{n-k-2\choose k-2}+1$), we conclude that the displayed holds trivially. Thus, \eqref{eqhm} is valid in any case.

If $\gamma(\ff)>{n-3\choose k-2}$ then we repeat the proof, but using the parts $\mathcal P_a,\mathcal P_b$ that contain the $k$-element sets that intersect $[4]$ in $\{1,3\}$ and in $\{2,4\}$, respectively.
\end{proof}

\subsection{Further generalizations}\label{sec34}
In this subsection, we give several further generalizations and strengthenings. It includes weighted versions and full cross-intersecting version of our results. All of them we state without proof or with a sketch of proof since the proofs use the same ideas as the proof of Theorems~\ref{thmfull2} and~\ref{thmfulleq}.

Our techniques allow us to give a weighted version of Theorems~\ref{thmfull1} and~\ref{thmfull2}. Assume that,  given a lower bound on $\gamma(\ff)$, we aim to maximise the expression  $\Delta(\ff)+c\gamma(\ff)$ with some $c>1$. (In terms of cross-intersecting families, we are maximising the expression $|\aaa|+c|\bb|$.) Then the following
is true.

\begin{thm}\label{thmfullw} Let $n>2k\ge 6$. Consider all resistant pairs $S_l\subset [n],\ T_l\subset [2,n]$, where $l=0,\ldots,m$. Assume that $T_0<T_1<T_2<\ldots <T_m$. Put $C:= \frac{n-k-2}{k-2}>1$.  Then
\begin{equation}\label{eqfull4} |\mathcal L(S_{l-1},k-1)|+C|\mathcal L(T_{l-1},k)|\ge|\mathcal L(S_{l},k-1)|+C|\mathcal L(T_{l},k)| \ \ \ \text{ for each }l \in [m],\end{equation}
and any cross-intersecting pair of families $\mathcal A\subset {[2,n]\choose k-1},\ \mathcal B\subset {[2,n]\choose k}$ with $|\mathcal L(T_{l-1},k)|<|\bb|\le |\mathcal L(T_l,k)|$ satisfies $|\aaa|+C|\bb|\le |\mathcal L(S_l,k-1)|+C|\mathcal L(T_l, k)|$.\vskip+0.1cm

In terms of intersecting families, if $\ff\subset {[n]\choose k}$ is intersecting and $|\mathcal L(T_{l-1},k)|<\gamma(\ff)\le |\mathcal L(T_l,k)|$, then $\Delta(\ff)+C\gamma(\ff)\le |\mathcal L(S_l,k-1)|+C|\mathcal L(T_l, k)|$.
\end{thm}

\begin{proof}[Sketch of the proof] The proof of this theorem follows the same steps as that of Theorem~\ref{thmfull2}. We sketch the proof of the cross-intersecting version of the theorem.
Using Lemma~\ref{lemdiv}, we may assume that $\aaa$ and $\bb$ form a resistant pair (indeed, otherwise, replacing $\aaa,\ \bb$ with $\aaa',$ $\bb'$ which satisfy $|\aaa'|+|\bb'|\ge |\aaa|+|\bb|$, $|\bb'|>|\bb|$ definitely increases the value of $|\aaa|+C|\bb|$). Then, looking at the proof of Lemma~\ref{lemres}, we see that in each of the cases the bipartite graph $G$ had parts of sizes ${n-i\choose z_a}$ and ${n-i\choose z_b}$, where $z_a+z_b = 2k-i$ and $z_a>z_b$. Since $i\ge 5$, we have $z_b\le k-3$. Therefore, if we put weight $w$, $w\le {n-i\choose z_a}/{n-i\choose z_b}$, on each vertex of the part $\mathcal P_b$ and weight 1 on each vertex of the part $\mathcal P_a$, we can still conclude that $\mathcal P_a$ is the independent set of the largest weight in $G$. The rest of the argument works out as before.
We have
\begin{equation}\frac{{n-i\choose z_a}}{{n-i\choose z_b}}\ge\frac{(n-i-z_b)}{z_b+1} = \frac{(n-2k+z_a)}{z_b+1}\ge \frac{(n-2k+z_b+1)}{z_b+1}\ge \frac{(n-k-2)}{k-2}.\end{equation}
Therefore, we may use $w=\frac{n-k-2}{k-2}$ as the weight of the vertices in $\mathcal P_b$.
\end{proof}

\begin{cor}\label{corweight}
 Let $n>2k\ge 6$. For any intersecting family $\ff\subset{[n]\choose k},$ $\gamma(\ff)\le {n-4\choose k-3}$, we have $|\Delta(\ff)|+\frac{n-k-2}{k-2}\gamma(\ff)\le {n-1\choose k-1}$.

 If, additionally, $\ff$ is non-trivial, then $|\Delta(\ff)|+\frac{n-k-2}{k-2}\gamma(\ff)\le {n-1\choose k-1}-{n-k-1\choose k-1}+\frac{n-k-2}{k-2}$.
\end{cor}
\vskip+0.2cm

Is not difficult to extend the considerations of Section~\ref{sec3} to the case of cross-intersecting families $\aaa\subset {[n]\choose a}, \bb\subset {[n]\choose b}$. The wording of Theorem~\ref{thmfull2} would stay practically the same. One just need to adjust Definition~\ref{defrespa}. Note that, unlike before in this section, here we will work with cross-intersecting families on $[n]$ (and not $[2,n]$), and so  $\mathcal L(S, a)$ stands for $\mathcal L([n],S,a)$.

\begin{defi} We say that two sets $S, T\subset [n]$  \underline{form an $(a,b)$-resistant pair}, if either $S = \{b-a+2\}, T =[b-a+2]$ or the following holds. Assuming that the largest element of $T$ is $j$, we have
\begin{enumerate}
  \item $S\cap T =\{j\}$, $S\cup T= [j]$, $|S|\le a$  and $|T|\le b$;
  \item for each $i\ge b-a+2$ we have $|[i]\cap S|-a< |[i]\setminus S|-b$.
\end{enumerate}
\end{defi}

  Let $m$ be the number of resistant pairs. For convenience, put $T_0 = [2,n]$ to correspond to the empty family, as well as $T_{m+1}=[b-a+1]$ and $T_{m+2}$ to be the analogues of the sets $\{2,3\}$ and $\{2,4\}$, respectively.
  Below we state a theorem, which is the analogue of Theorems~\ref{thmfull2},~\ref{thmfulleq},~\ref{thmfullw} and (the second part of) Corollary~\ref{corhm} in the case of general cross-intersecting families. Its proof is a straightforward generalization of the proofs of the respective theorems, and thus we omit it.

\begin{thm}\label{thmfullcri} Let $b, a>0$, $n>a+b$. Consider all $(a,b)$-resistant pairs $S_l\subset [n],\ T_l\subset [2,n]$, where $l\in [m]$. Assume that $T_0<T_1<T_2<\ldots T_m$.

1. Then
\begin{equation}\label{eqfull5} |\mathcal L(S_{l-1},a)|+|\mathcal L(T_{l-1},b)|>|\mathcal L(S_{l},a)|+|\mathcal L(T_{l},b)| \ \ \ \text{ for each }l \in [m],\end{equation}
and any cross-intersecting pair of families $\mathcal A\subset {[n]\choose a},\ \mathcal B\subset {[n]\choose b}$ with $|\mathcal L(T_{l-1},b)|<|\bb|\le |\mathcal L(T_l,b)|$ satisfies \begin{equation}\label{eqfull6} |\aaa|+|\bb|\le |\mathcal L(S_l,a)|+|\mathcal L(T_l, b)|.\end{equation}
If the families $\mathcal L(|\aaa|,a)$, $\mathcal L(|\bb|,b)$ have characteristic sets $S,T$, then we have equality in \eqref{eqfull6} if and only if $S,T$ is a $T_l$-neutral $(a,b)$-pair, where the notion of a $T_l$-neutral $(a,b)$-pair is a straightforward generalization of that of neutral pair.

2. The same conclusion holds with $|\mathcal B|$, $|\mathcal L(T_{l-1},b)|$ $|\mathcal L(T_l,b)|$ replaced with $C|\mathcal B|$, $C|\mathcal L(T_{l-1},b)|$ and $C|\mathcal L(T_l,b)|$, where $C$ is a constant, $C<\frac{n-b-1}{a-1}.$

3. Denote $t:=b+1-a$. Assume that ${n-i\choose b-i}<|\bb|\le {n-t\choose a-1}+{n-t-1\choose a-1}$ for integer $i\in [t+1,b]$. If $i\ge t+2$ then we have \begin{small}$$|\aaa|+|\bb|\le {n\choose a}-{n-i\choose a}+{n-i\choose b-i} -{n-b-1\choose a-1}+1.$$\end{small}

If $i=t+1$ and $|\bb|\notin [{n-t\choose a-1}-{n-b-1\choose a-1}+2,{n-t\choose a-1}]$, $|\bb|\le {n-t\choose a-1}+{n-t-1\choose a-1}-{n-b-1\choose a-1}+1$ then we have \begin{small}$$|\aaa|+|\bb|\le {n\choose a}-{n-t\choose a}+{n-t\choose b-t} -{n-b-1\choose a-1}+1.$$\end{small}
\end{thm}
Remark that we have $|\aaa|+|\bb|={n\choose a}-{n-t\choose a}+{n-t\choose a-1}$ for $\aaa, \bb$ defined by characteristic sets $S_i, T_i$, where $i=m,m+1,m+2$.

This theorem generalizes and strengthens many results on cross-intersecting families, in particular, the theorem for cross-intersecting families proven in \cite{KZ} and the following theorem due to Frankl and Tokushige \cite{FT}

\begin{thm}[Frankl, Tokushige, \cite{FT}]\label{lemft} Let $n > a+b$, $a\le b$, and suppose that families $\mathcal F\subset{[n]\choose a},\mathcal G\subset{[n]\choose b}$ are cross-intersecting. Suppose that for some real number $\alpha \ge 1$ we have  ${n-\alpha \choose n-a}\le |\mathcal F|\le {n-1\choose n-a}$. Then \begin{equation}\label{eqft}|\mathcal F|+|\mathcal G|\le {n\choose b}+{n-\alpha \choose n-a}-{n-\alpha \choose b}.\end{equation}
\end{thm}


One easy corollary of \eqref{eqfull5} and \eqref{eqfull6}, which also appeared in \cite{KZ} and other places, is as follows:
\begin{cor}[\cite{KZ}]\label{corkz}
  Let $a,b>0$, $n>a+b$. Let $\aaa\subset {[n]\choose a},\ \bb\subset {[n]\choose b}$ be a pair of cross-intersecting families. Denote $t:=b+1-a.$ Then, if $|\bb|\le {n-t\choose a-1}$, then
\begin{equation}\label{eqcreasy} |\aaa|+|\bb|\le {n\choose a}.\end{equation}
Moreover, the displayed inequality is strict unless $|\bb|=0$.

If ${n-j\choose b-j}\le |\bb|\le {n-t\choose a-1}$ for integer $j\in [t, b]$, then
\begin{equation}\label{eqcreasy2} |\aaa|+|\bb|\le {n\choose a}-{n-j\choose a}+{n-j\choose b-j}.\end{equation}
Moreover, if the left inequality on $\bb$ is strict, then the inequality in the displayed formula above is also strict, unless $j=t+1$ and $|\bb| = {n-t\choose a-1}$.
\end{cor}
We note that the results in \cite{KZ} did not explicitly treat the equality case. However, it is clear that strictness of \eqref{eqcreasy} follows from \eqref{eqcreasy2}, and the equality case in \eqref{eqcreasy2} follows from Theorem~\ref{thmfullcri} part 3.

\section{Beyond Hilton--Milner. Proofs of Theorem~\ref{thmhk} and~\ref{thmclass2}}\label{sec4}
\begin{proof}[Proof of Theorem~\ref{thmhk}]
As we have already mentioned in the introduction, applying Corollary~\ref{corhm} with $u=k$, we conclude that \eqref{eqhk} holds for $k\ge 4$.

Next, in terms of Theorem~\ref{thmfull1}, we know that $\gamma_i = i$ for $i\in[k-3]$, and $\gamma_{k-2} = n-k$. Thus, for $k\ge 5$, $\gamma_2= 2$ and, using Theorem~\ref{thmfull1}, we conclude that any intersecting family $\ff\subset {[n]\choose k}$ with $\gamma(\ff)\ge 3$ is strictly smaller than the right hand side of \eqref{eqhk}. Therefore, if $\ff$ with $\gamma(\ff)\ge 2$ has size equal to the right hand side of \eqref{eqhk}, then $\gamma(\ff)=2$. But then any maximal $\ff$ must be isomorphic to $\mathcal J_i$ for some $i\ge 2$ (since the family is uniquely determined by the size of the intersection of the two sets contributing to the diversity). Finally, as we have mentioned in the introduction, $|\mathcal J_i|<|\mathcal J_2|$ for any $i>2$.
\end{proof}

Let us  recall the definition of shifting. For a given pair of indices $1\le i<j\le n$ and a set $A \subset [n]$, define its \underline{$(i,j)$-shift} $S_{ij}(A)$ as follows. If $i\in A$ or $j\notin A$, then $S_{ij}(A) = A$. If $j\in A, i\notin A$, then $S_{ij}(A) := (A-\{j\})\cup \{i\}$. That is, $S_{ij}(A)$ is obtained from $A$  by replacing $j$ with $i$.
The  $(i,j)$-shift $S_{ij}(\mathcal A)$ of a family $\mathcal A$ is as follows:
$$S_{ij}(\mathcal A) := \{S_{ij}(A)\ :\  A\in \mathcal A\}\cup \{A\ :\  A,S_{ij}(A)\in \mathcal A\}.$$

\subsection{Proof of Theorem~\ref{thmclass2}}\label{sec41}
Let us deal with the first part of the statement first. Assume that the statement does not hold and choose $\ff$, $t$, $\mm$, and $\ff'$ satisfying the requirements of the theorem, so that $|\ff|>|\ff'|$ and $\ff$ has the smallest possible diversity under these conditions.
First, we can clearly suppose that  $\gamma(\ff)\le {n-4\choose k-3}$. Indeed, the largest family with diversity bigger than ${n-4\choose k-3}$ is at most as large as $\mathcal H_4$ (cf. \eqref{eqhu}), and $\mathcal H_4$ contains  a copy of any $\G$ as in the statement of the theorem. Therefore, by the choice of $\ff$, it cannot have diversity larger than $\mathcal H_4$, otherwise we should have replaced $\ff$ with $\mathcal H_4$.

From now on, we suppose that $\gamma(\ff)\le {n-4\choose k-3}$. Suppose that $\bigcap_{M\in \mm}M = [2,t+1]$.

For $i=2,\ldots,t+1$, we may consecutively apply all the $S_{ij}$-shifts, where $j>i$, to $\ff$. Note that the family $\mm$ stays intact under these shifts, and the diversity and size of $\ff$ is not affected. Thus, we may assume that $\ff$ is invariant under these shifts. For any $j\ge 2$ and $S\subset [2,j]$, define $$\ff(S,[j]):=\big\{F\subset [j+1,n]\ :\  F\cup S\in \ff(\bar 1)\big\}$$ and, for any family $\mathcal R\subset {[n]\choose k}$, let $\partial\mathcal R:=\bigcup_{R\in \mathcal R} {R\choose k-1}$ denote the \underline{shadow} of $\mathcal R$. (We gave an equivalent definition in Theorem~\ref{thmkk}.)

An important consequence of the shifts we made is that, for any $i\in [2,t+1]$ and $S\subset [2,i-1]$, we have \begin{equation}\label{eqcard} \big|\ff(S\cup \{i\},[i])\big|\ge \big|\partial \ff(S,[i])\big|.\end{equation} Actually, we have $\ff(S\cup \{i\},[i])\supset \partial \ff(S,[i])$.  Indeed, for any $F'\in \partial \ff(S,[i])$, there exists $j\in [i+1,n]$ and $F\in \ff(S,[i])$, such that $F'\cup \{j\}=F$, and, since we the $(i,j)$--shift, we have $F'\cup \{i\}\cup S\in \ff$ and $F'\in \ff(S\cup \{i\},[i])$.

We have $|\ff(\bar 1)|\le {n-4\choose k-3}$ by assumption. Let us show that, for any $i\in [2,4]$, we have \begin{equation}\label{eqcard2} |\ff([2,i-1],[i])|\le {n-5\choose k-3}.\end{equation} (Naturally, we put $[2,1]=\emptyset$.)
Assume that \eqref{eqcard2} does not hold. Then, using the Kruskal--Katona theorem, it is not difficult to see that $|\partial \ff([2,i-1],[i])|> {n-5\choose k-4}$,\footnote{Assume that we have equality in \eqref{eqcard2}. The Kruskal--Katona theorem in Lovasz' form states that if $\mathcal K$ a family of $m$-element sets and, for some real $x\ge m$, we have $|\mathcal K|\ge {x\choose m}$, then $|\partial \mathcal K|\ge {x\choose m-1}$. In our case, $\mathcal K:=\ff([2,i-1],[i])\subset {[i+1,n]\choose k-i+1}$ and ${n-5\choose k-3}={x\choose k-i+1}$, where $x\le n-5$ for any $i\in [2,4].$ Finally, we have $\frac{{x\choose k-i+1}}{{x\choose k-i}}=\frac {x-k+i}{k-i+1}\le \frac {n-k-1}{k-3}=\frac{{n-5\choose k-3}}{{n-5\choose k-4}}$, and thus $|\partial \mathcal K|> {x\choose k-i}\ge {n-5\choose k-4}$.
}
and thus
$$\big|\ff(\bar 1)\big|\ge \big|\ff([2,i-1],[i])\big|+\big|\ff([2,i],[i])\big|\overset{\eqref{eqcard}}{\ge} \big|\ff([2,i-1],[i])\big|+\big|\partial \ff([2,i-1],[i])\big|>{n-4\choose k-3},$$
a contradiction. For each $i\in [2,4]$, consider the following bipartite graph $G_i$. The parts of $G_i$ are
\begin{align*}
\mathcal P_a(i) :=\ &\Big\{P\ :\ i\in P\in {[2,n]\choose k-1}\Big\},\\
\mathcal P_b(i) :=\ &\Big\{P \ :\ i\notin P \in {[2,n]\choose k}\Big\},
\end{align*}
and edges connect disjoint sets. We identify $\mathcal P_a(i)$ with ${X\choose k-2}$ and $\mathcal P_b(i)$ with ${X\choose k}$, where $X := [2,n]\setminus \{i\}$, $|X|=n-2>k+k-2$.

Due to \eqref{eqcard2}, we have $|\mathcal P_b(i)\cap \ff(\bar 1)|\le {n-5\choose k-3}={|X|-3\choose (k-2)-1}$ for $i=2$. Thus, we can apply \eqref{eqcreasy} to
\begin{align*}
\aaa:=\ff(1)\cap \mathcal P_a(i) \ \ \ \ \text{and}\ \ \ \
\bb:=\ff(\bar 1)\cap \mathcal P_b(i),
\end{align*}
where $i=2$, and conclude that $|\aaa|+|\bb|\le {|X|\choose k-2}$ (and the inequality is strict unless $\bb=\emptyset$). Therefore, removing $\ff(\bar 1)\cap \mathcal P_b(2)$ from $\ff(\bar 1)$ and adding sets from $\mathcal P_a(2)$ to $\ff(1)$, we get a pair of families with larger sum of cardinalities. Moreover, the new pair is cross-intersecting: all sets in $\ff(\bar 1)\setminus \mathcal P_b(2)$, as well as the sets from $\mathcal P_a(2)$, contain $2$. Thus, by our choice of $\ff$ we may assume that $\bb=\emptyset$ and so all sets in $\ff(\bar 1)$ must contain $2$. Repeat this argument for $i=3$ and $i=4$.
Concluding, we may assume that all sets in $\ff(\bar 1)$ contain $[2,4]$.

Next, for each $i\ge 5$, consider the following slightly different bipartite graph $G_i'$. The parts of $G_i$ are
\begin{align*}
\mathcal P_a'(i):=\ &\Big\{P\ :\  P\in {[2,n]\choose k-1},\ P\cap [2,i]=\{i\}\Big\},\\
\mathcal P_b'(i):=\ &\Big\{P\ :\  P\in {[2,n]\choose k},\ P\cap [2,i]=[2,i-1]\Big\},
\end{align*}
and edges connect disjoint sets. We identify $\mathcal P_a'(i)$ with ${[i+1,n]\choose k-2}$ and $\mathcal P_b'(i)$ with ${[i+1,n]\choose k-i+1}$.

We have $|\mathcal P_b'(i)\cap \ff(\bar 1)|\le {n-i\choose k-i+1}\le {n-5\choose k-3}$. Thus, for each $i=5,\ldots, t+1$ we can apply \eqref{eqcreasy} to
\begin{align*}
\aaa:=\ff(1)\cap \mathcal P_a'(i) \ \ \ \ \text{and}\ \ \ \ \
\bb:=\ff(\bar 1)\cap \mathcal P_b'(i),
\end{align*}
and conclude that $|\aaa|+|\bb|\le {n-i\choose k-2}$ (and the inequality is strict unless $\bb=\emptyset$). Arguing as before, we  may assume that all sets in $\ff(\bar 1)$ contain $[2,t+1]$. \vskip+0.1cm

Put $\mm = \{M_1,\ldots, M_z\}$. Since $\mm$ is minimal, for each $M_l\in \mm$, $l\in [z]$, there is \begin{equation}\label{eqil} i_l\in \Big(\bigcap_{M\in \mm\setminus \{M_l\}}M\Big)\setminus \Big(\bigcap_{M\in \mm}M\Big).\end{equation}
We assume that $i_l = t+l+1$, $l\in[z]$. In particular, $\{i_1,\ldots, i_z\} = [t+2,t+z+1]$.

Next, for each $i=t+2,\ldots, t+z+1$ consider the bipartite graph $G'_i$, defined above. We can apply \eqref{eqcreasy2} with $k-2,k-i+1, k-i+1, 4-i$ playing roles of $a,b,j,t$, respectively. Indeed, we know that $|\bb|\ge {n-i\choose (k-i+1)-(k-i+1)}=1$ since $M_{i-t-1}\in \mathcal P_b'(i)$ (note that $M_l\notin \mathcal P_b'(i)$ for $l\ne i-t-1$, since all of them contain $i_l$ due to the definition of $i_l$). Therefore, $|\mathcal A|+|\mathcal B|\le {n-i\choose k-2}-{n-k-i\choose k-2}+1$. For $k\ge 5$, the inequality is strict unless $\bb=\{M_{i-t-1}\}$. For $k=4$, we apply it for $i=4$, thus $n-5,2,1,1,0$ play the roles of $n,a,b,j,t$, respectively, and so the inequality is strict unless $|\bb| = {n-5\choose 1}$, i.e., $\bb$ contains all possible $1$-element sets. Thus,  we may replace $\aaa,\ \bb$ with $\bb':=\{M_{i-t-1}\}$ and $\aaa':=\{A\in {[i+1,n]\choose k-2}\ :\  A\cap M_{i-t-1}\ne \emptyset\}$. The resulting family is cross-intersecting. Thus, by our choice of $\ff$, all sets in $\ff(\bar 1)\setminus \mm$ must contain $i$. Iterating this procedure for  $i=t+2,t+3,\ldots, t+z+1$, we may assume that any set in $\ff(\bar 1)\setminus \mm$ must contain the set $[2, t+z+1]$.
Put $t':=t+z+1$. If $\ff(\bar 1)\setminus \mm$ is non-empty then, for each $l \in [z]$, the set $M_l\setminus [2,t']$  must be non-empty: otherwise, $t'-1>k$. If $\ff(\bar 1)\setminus \mm$ is empty then $\ff(\bar 1)=\mm$, which contradicts $|\ff|>|\ff'|$. Note that $t'-1\ge z+3$.

For each $l\in [z]$, select
one element $i_l\in M_l\cap [t'+1,n]$. Note that $i_l$ may coincide. Put $I:=\{i_l\ :\ l\in [z]\}$. Consider the bipartite graph $G(t',I)$ with parts
\begin{align*}
\mathcal P_a(t',I):=\ &\Big\{P \ :\ P\in {[2,n]\choose k-1},\ I\subset P,\  [2,t']\cap P=\emptyset\Big\},\\
\mathcal P_b(t',I):=\ &\Big\{P\ :\ P\in {[2,n]\choose k},\ [2,t']\subset P,\ I\cap P=\emptyset\Big\},
\end{align*}
and edges connecting disjoint sets. We identify $\mathcal P_a(t',I)$ with ${Y\choose k-z'}$, $z'\le z+1$, and $\mathcal P_b(t',I)$ with ${Y\choose k-z''}$, $z''=t'-1\ge z+3$, where $Y = [t'+1,n]\setminus I$, $|Y|=n-z''-z'$. In particular, $|Y|>k-z'+k-z''$. By the choice of $I$, we have $|\mm\cap \mathcal P_b(t',I)| =\emptyset$. 
Denote
\begin{align*}
\aaa:=\ff(1)\cap \mathcal P_a(t',I) \ \ \ \ \text{and}\ \ \ \ \
\bb:=\ff(\bar 1)\cap \mathcal P_b(t',I).
\end{align*}
We have $k-z'>k-z''$, and, therefore, we may apply  \eqref{eqcreasy} with $a:=k-z',\ b:=k-z'',\ j:=k-z''$ (the upper bound on $|\bb|$ becomes trivial in that case) and conclude that $|\aaa|+|\bb|\le {|Y|\choose k-z'}$ (and the inequality is strict unless $\bb=\emptyset$).
As before, replacing $\aaa$ with $\mathcal P_a(I)$ and $\bb$ with $\emptyset$ does not decrease the sum of sizes of the families and preserves the cross-intersecting property.
Thus, by the choice of $\ff$, we must have $\bb=\emptyset$.

Repeating this for all possible choices of $I$, we arrive at the situation when any set from $\ff(\bar 1)\setminus\mm$ must intersect {\it any} such set $I$. Clearly, this is only possible for a set $F$ if $F\supset M_l\cap[t'+1,n]$. But this implies that $|F|>|M_l|$, which is impossible. Thus $\ff(\bar 1) = \mm$, which contradicts $|\ff|>|\ff'|$, so the proof of \eqref{eqclass1} is complete.

Finally, uniqueness follows from the fact that the inequalities \eqref{eqcreasy}, \eqref{eqcreasy2} are strict unless the family $\bb$ has sizes $0$ and $1$, respectively, or $k=4$ (we mentioned it at every application). Therefore, if $\ff(\bar 1)\ne \mm$, then at some point we would have had a strict inequality in the application of \eqref{eqcreasy}, \eqref{eqcreasy2}.\\

Let us now prove the moreover part of the statement of Theorem~\ref{thmclass2}. First, if there is no family $\mathcal M\subset \ff(\bar 1)$, minimal w.r.t. common intersection and such that $|\bigcap_{M\in\mm}M|=t$, then we apply $(i,j)$-shifts to $\ff$ for $2\le i<j\le n$ until such family appears. Since common intersection of any subfamily of $\ff(\bar 1)$ may change by at most one after any shift, either we obtain the desired $\mm$ or we arrive at a shifted family $\ff(\bar 1)$ without such $\mm$. But the latter is impossible. Indeed, for a shifted intersecting family $\ff(\bar 1)$, we have $|\bigcap_{F\in \ff(\bar 1)}F|=[2,j]$ for some $j\ge 2$, and in that case $[2,j]\cup [j+2,k+2]\in \ff(\bar 1)$. But then the set $F_{j'}:=[2+j']\cup [j'+2,k+2]\in\ff(\bar 1)$, where $j\le j'\le k+1$. It is clear that the sets $F_j,\ldots, F_{k-t+j+1}$ form a subfamily that has common intersection of size $t$.

Fix $\mathcal M\subset \ff(\bar 1)$, minimal w.r.t. common intersection, such that $|\bigcap_{M\in\mm}M|=t$. Applying the first part of Theorem~\ref{thmfull2}, we may assume that $\ff(\bar 1)=\mm$. Clearly,  the number of sets in $\ff(1)$ passing through $\bigcap_{M\in\mm}M$ is always the same, independently of the form of $\mm$. Thus, we need to analyze the sum of sizes of the family $\ff':=\big\{F\in \ff(1)\ :\  F\cap \bigcap_{M\in\mm}M=\emptyset\big\}$ and $\mm':= \big\{M\setminus \bigcap_{M\in\mm}M\ :\  M\in \mm\big\}.$ Note that $\mm'$ and $\ff'$ are cross-intersecting, moreover, $\tau(\mm')=2$ and $\mm'$ is minimal w.r.t. this property.

The following lemma concludes the proof of the theorem. Let us first give some definitions.
Given integers $m>2s$, let us denote by $\mathcal T_2'(s):=\{[s], [s+1,2s]\}$. Let $\ff_2'(s)\subset {[m]\choose k-1}$ stand for the largest family, cross-intersecting with $\mathcal T_2'(s)$. Let $\ff_2(s)\subset {[m]\choose k-1}$ stand for the largest family, cross-intersecting with $\mathcal T_2(s)$  (cf. \eqref{deft2}).

\begin{lem}\label{lemmin} Let $k\ge s$ and $m\ge k+s$ be integers, $k\ge 4$. Given a family $\mathcal H\subset {[m]\choose s}$ with $\tau(\mathcal H)=2$ and minimal w.r.t. this property, consider  the maximal family $\ff\subset {[m]\choose k-1}$ that is cross-intersecting with $\mathcal H$. Then the unique maximum of $|\ff|+|\mathcal H|$ is attained when $\mathcal H$ is isomorphic to $\mathcal T_2'(s)$ (and $\ff$ is thus isomorphic to $\ff_2'(s)$).

If we additionally require that $\mathcal H$ is intersecting\footnote{Note that this is equivalent to requiring that $|\mathcal H|>2$.} then the maximum of $|\ff|+|\mathcal H|$ is attained for $\mathcal H$ and $\ff$ isomorphic to $\mathcal T_2(s)$ and $\ff_2(s)$. The maximal configuration is unique if $s\ge k$.
\end{lem}

We can apply the first part of Lemma~\ref{lemmin} in our situation with $[2,n]\setminus \bigcap_{M\in\mm}M$ playing the role of $[m]$ and $k-t$ playing the role of $s$. Note that $|[2,n]\setminus \bigcap_{M\in\mm}M| = n-t-1\ge 2k-t$, and so the condition on $m$ from the lemma is satisfied. This proves \eqref{eqclass2}, moreover, we get that for $k\ge 5$ the inequality was strict unless $|\ff(\bar 1)|=2$ in the first place. But then $\ff$ is isomorphic to a subfamily of $\mathcal J_{i}$ for $i\ge i'$, and the equality is possible only if $\ff$ is isomorphic to $\mathcal J_i$.
Thus, to conclude the proof of Theorem~\ref{thmfull2}, we only have to prove the lemma.

We note that the second, more complicated, part of Lemma`\ref{lemmin} is not needed for this application, however, it will be crucial for the completion of the proof of Theorem~\ref{thmtau3}.

\begin{proof}[Proof of Lemma~\ref{lemmin}]
Let us first express $|\ff'_2(s)|$. It is not difficult to see that
\begin{footnotesize}\begin{align}
\notag  |\ff'_2(s)|\  =\ & {m-1\choose k-2}-{m-s-1\choose k-2} + \\
\notag   & {m-2\choose k-2}-{m-s-2\choose k-2}+\\
\notag   &\cdots \\
\label{eqfs}   &{m-s\choose k-2}-{m-2s\choose k-2}.
\end{align}\end{footnotesize}
Indeed, in the first line we count the sets containing $1$ that intersect $[s+1,2s]$, in the second line we count the sets not containing $1$, containing $2$ and intersecting $[s+1,2s]$ etc.

We can actually bound the size of $\ff$ for any $\mathcal H$ in a similar way. Suppose that $z:=|\mathcal H|$ and $\mathcal H = \{H_1,\ldots, H_z\}$. Since $\mathcal H$ is minimal, for each $l\in[z]$ there exists an element $i_l$ such that $i_l\notin H_l$ and $i_l\in \bigcap_{j\in[z]\setminus \{l\}} H_l$. (All $i_l$ are of course different, cf. also \eqref{eqil}.) Applying Bollobas' set-pairs inequality \cite{Bol} to $\mathcal H$ and $\{i_{l}\ :\ l\in[z]\}$, we get that $|\mathcal H|\le {s+1\choose s}=s+1$.

For each $l=2,\ldots, z$, we count the sets $F\in\ff$ such that $F\cap \{i_2,\ldots, i_l\}=\{i_l\}$. Such sets must additionally intersect $H_{l}\setminus \{i_2,\ldots, i_{l-1}\}$. Note that $H_1\supset \{i_2,\ldots, i_z\}$.
Assuming that $H_1\setminus \{i_2,\ldots,i_z\} = \{j_1,\ldots, j_{s-z}\}$, for each  $l=1,\ldots, s-z$ we further count the sets $F\in \ff$ such that $F\cap \{i_2,\ldots, i_z, j_1,\ldots, j_l\}=\{j_l\}$. Such sets must additionally intersect $H_{i}\setminus \{i_2,\ldots, i_{z}\}$ for some $i\in [2,z]$.\footnote{Otherwise, $\mathcal H$ would have common intersection. Note that $H_{i}\setminus \{i_2,\ldots, i_{z}\}$ is a set of size $s-z+2$.}
 Since $F\cap H_1\ne \emptyset$ for any $F\in \ff$ and given that the classes for different $l$ are disjoint, we clearly counted each set from $\ff$ exactly once. (However, we may also count some sets that are not in $\ff$.) Doing this count, we get the following bound on $\ff$.
\begin{footnotesize}\begin{align}
 \notag |\ff|\ \le \ & {m-1\choose k-2}-{m-s-1\choose k-2} + \\
\notag   & {m-2\choose k-2}-{m-s-1\choose k-2}+\\
\notag   &\cdots\\
\notag   & {m-z+1\choose k-2}-{m-s-1\choose k-2}+\\
\notag   & {m-z\choose k-2}-{m-s-2\choose k-2}+\\
\notag   &\cdots \\
\label{eqfz}   &{m-s\choose k-2}-{m-2s-2+z\choose k-2}\ =:\ f(z).
\end{align}\end{footnotesize}
Remark that \eqref{eqfz} coincides with \eqref{eqfs} when substituting $z=2$.
We have $f(z-1)-f(z)\ge {m-s-1\choose k-2}-{m-s-2\choose k-2}={m-s-2\choose k-3}> 1$ (here we use that $m\ge s+k$ and $k\ge 4$). Therefore, for any $z\ge z'$, \begin{equation}\label{eqz'}|\mathcal H|+|\ff|\le f(z')+z',\end{equation} and the inequality is strict unless $z=|\mathcal H|=z'$.

At the same time, we have $|\ff'_2(s)|+|\mathcal T'_2(s)|=f(2)+2$ and $|\ff_2(s)|+|\mathcal T_2(s)|=f(3)+3$! (The former we have seen above, and the latter is easy to verify by doing exactly the same count.) Since, up to isomorphism, there is only one family $\mathcal H\subset {[m]\choose s}$ of size $2$ with $\tau(\mathcal H)=2$, we immediately conclude that the first part of the statement holds. To deduce the second part, we only need to show that, among all possible choices of $\mathcal H$ of size $3$, the only one (up to isomorphism) that attains equality in \eqref{eqz'} is $\mathcal H = \mathcal T_2(s)$.

Recall that, for uniqueness in the second part of the lemma, we have additional condition $s\ge k$.
If there are two sets $H',H''\in \mathcal H$ such that $|H'\cap H''|=s-1$, then $\mathcal H$ is isomorphic to $\mathcal T_2(s)$. Therefore, in what follows we assume that $|H'\cap H''|\le s-2$ for any $H',H''\in\mathcal H$.

Let us deal with the case when $H_{l}\cap H_{l'} = i_{l''}$ for any $\{l,l',l''\}=[3]$. Note that this implies that
\begin{equation}\label{eqbign}
  m\ge 3s-3.
\end{equation} Since $s\ge k \ge 4$, there are elements $j_l\in H_l\setminus (H_{l'}\cup H_{l''})$, $\{l,l',l''\}=[3]$. Perform the $(j_1,j_2)$-shift on $\ff\cup \mathcal H$ and denote $\ff':=S_{j_1j_2}(\ff)$. Clearly, the sizes of the families stay the same and the resulting families are cross-intersecting. The family, $S_{j_1j_2}(\mathcal H)$  has covering number $2$. Moreover, we can add new sets to $\ff'$ without violating the cross-intersecting property. Indeed, consider the families
\begin{align*}\mathcal A\ :=&\ \Big\{F\in {[m]\choose k-1}\ :\  j_3\in F,\ F\cap (H_1\cup H_2)=\{j_2\}\Big\},\\
\mathcal A'\ :=&\ \Big\{F\in {[m]\choose k-1}\ :\  j_3\in F,\ F\cap (H_1\cup H_2)=\{j_1\}\Big\}.\end{align*}
It is easy to see that actually $\aaa' = S_{j_1j_2}(\aaa)$. Moreover, $\aaa\cap \ff = \emptyset$ since sets from $\aaa$ do not intersect $H_1$ and $\aaa'\cap \ff' = \emptyset$ since $\aaa'\cap \ff=\emptyset$ (again, since sets from $\aaa'$ do not intersect $H_2$) and $\aaa\cap \ff = \emptyset$. At the same time, $\aaa'$ may be included into $\ff'$ since sets from $\aaa'$ intersect all sets in $S_{j_1j_2}(\mathcal H)$. Finally, $\aaa = {[m]\setminus X\choose k-3}$, where $|X| = 2s$, and thus $|\aaa|\ge {k-3\choose k-3}$ due to $m\ge 3s-3\ge k+2s-3$, so $\aaa$ is non-empty, and thus $\mathcal H$ was not optimal.

Finally, we may assume that $|H_1\cap H_2|\in [2,s-2]$. Then we do the a similar count as for \eqref{eqfz}. The first two steps (with $i_2,i_3$) are the same. The part with $j_i$ is, however, slightly modified. Take $j'\in (H_1\cap H_2)\setminus \{i_3\}$ and $j'' \in H_1\setminus (H_2\cup \{i_2\}$. Such choices are possible due to $|H_1\cap H_2|\in [2,s-2]$. Count the sets $F\in \ff$ such that $F\cap \{i_2,i_3,j'\}=j'$. They must intersect $H_3\setminus \{i_2\}$. Next, crucially, count the sets in $F\in \ff$ such that $F\cap \{i_2,i_3,j',j''\}=j''$. They must intersect $H_2\setminus \{i_2,j'\}$ (note the size of this set is $s-2$ instead of $s-1$). The remaining count is the same: let $\{j_1,\ldots, j_{s-4}\}:=H_1\setminus \{i_2,i_3,j',j''\}$ and, for each $l\in[s-4]$, count the sets $F\in \ff$ such that $F\cap \{i_2,i_3,j',j'',j_1,\ldots, j_{l}\} =j_l$. They must additionally intersect either $H_2\setminus \{i_2,j'\}$, or $H_3\setminus \{i_3\}$. Thus, we obtain the following bound.
\begin{footnotesize}\begin{align}
 \notag |\ff|\ \le \ & {m-1\choose k-2}-{m-s-1\choose k-2} + \\
\notag   & {m-2\choose k-2}-{m-s-1\choose k-2}+\\
\notag   & {m-3\choose k-2}-{m-s-2\choose k-2}+\\
\notag   & {m-4\choose k-2}-{m-s-\textbf{2}\choose k-2}+\\
\notag   & {m-5\choose k-2}-{m-s-4\choose k-2}+\\
\notag   &\cdots\\
\label{eqfz}   &{m-s\choose k-2}-{m-2s+1\choose k-2}\ =:\ f'(3).
\end{align}\end{footnotesize}
We have $f(3)-f'(3) = {m-s-2\choose k-2}-{m-s-3\choose k-2} = {m-s-3\choose k-3}\ge 1$ due to $m\ge s+k$, and thus $|\ff|\le f'(3)<f(3) = |\ff_2(s)|$. Thus, in the assumption $s\ge k$ and if $\mathcal H$, $|\mathcal H|\ge 3$, is not isomorphic to $\mathcal T_2(s)$, we  have strict inequality in \eqref{eqz'} for $z'=3$. The lemma is proven.
\end{proof}

\section{Families with fixed covering number and maximum degree proportion}\label{sec5}

One of our main tools  for this section is the following structural result due to Dinur and Friedgut \cite{DF}. We say that a family $\mathcal J\subset 2^{[n]}$ is a {\it $j$-junta}, if there exists a subset $J\subset [n]$ of size $j$ (the {\it center} of the junta), such that the membership of a set in $\ff$ is determined only by its intersection with $J$, that is, for some family $\mathcal J^*\subset 2^{J}$ (the {\it defining family}) we have $\ff=\{F\ :\ F\cap J\in \mathcal J^*\}$.

\begin{thm}[\cite{DF}]\label{thmdf} For any integer $r\ge 2$, there exist functions $j(r), c(r)$, such that for any integers $1 < j(r) < k < n/2$, if $\ff\subset {[n]\choose k}$ is an intersecting family with $|\ff|\ge c(r){n-r\choose k-r}$, then there exists an intersecting $j$-junta $\mathcal J$ with
$j\le j(r)$ and \begin{equation}\label{eqDF} |\ff\setminus\mathcal J|\le c(r){n-r\choose k-r}.\end{equation}
\end{thm}

\subsection{Proof of Theorem~\ref{thmtau3}}\label{sec51}
Recall the expression of the size of $\mathcal C_3(n,k)$, obtained in the proof of Lemma~\ref{lemmin} (cf. \eqref{eqfz}):
\begin{footnotesize}
\begin{align}\label{sizec}
 \notag |\mathcal C_3(n,k)|\ =\ 3\ +\ &{n-2\choose k-2}-{n-k-2\choose k-2} + \\
 \notag  & {n-3\choose k-2}-{n-k-2\choose k-2}+\\
\notag   & {n-4\choose k-2}-{n-k-3\choose k-2}+\\
 \notag  &\cdots \\
   &{n-k-1\choose k-2}-{n-2k\choose k-2}.
\end{align}
\end{footnotesize}

Let $C$ be a sufficiently large constant, which value shall be clear later.\footnote{We use a subscript ``$C$'' in the inequalities, i.e., $\ge_C$, in which we need that $C$ is sufficiently large.} The case of $k\le C$ follows from the original result of Frankl. In what follows, we assume that $k\ge C$. Take any family $\ff\subset {[n]\choose k}$ with $\tau(\ff)=3$. Theorem~\ref{thmdf} implies that there exists a set $J$, $|J|\le c(5)$, and an intersecting family $\mathcal J^*\subset 2^{[J]}$, such that $|\ff \setminus \mathcal J|\le j(5){n-5\choose k-5}\le_C {n-5\choose k-4}$, where $\mathcal J:=\{F\in {[n]\choose k}\ :\  F\cap J\in \mathcal J^*\}$.

The first step of the proof is to show that $\mathcal J$ is, in fact, a dictatorship: $\mathcal J^*$ consists of one singleton. Indeed, if $\mathcal J^*$ contains a singleton, then we may as well assume that $\mathcal J^*$ consists of that only singleton. Otherwise,
any set in $\mathcal J\cap \ff$ must intersect $J$ in at least $2$ elements. Moreover, for any two elements we choose, there is a set in $F\in\ff$ that does not contain any of those two elements, so any set in $\mathcal J\cap \ff$ must intersect $F$. This gives the bound
\begin{equation}\label{eqsizehun}|\mathcal J\cap \ff|\le {|J|\choose 2}\Big({n-2\choose k-2}-{n-k-2\choose k-2}\Big),\end{equation}
and, given that
$${n-|J|^2\choose k-2}-{n-|J|^2-k+1\choose k-2}>_C 0.99\Big({n-2\choose k-2}-{n-k-2\choose k-2}\Big),$$
 we compare  \eqref{eqsizehun} with \eqref{sizec} and get that $|\mathcal J\cap \ff|<\frac 12 |\mathcal C_3(n,k)|$ (remark that we also used that $k>C$ and thus that the number of lines in \eqref{sizec} is  bigger than $|J|^2$). Since  $|\ff\setminus \mathcal J|\le {n-5\choose k-4}<\frac 12 |\mathcal C_3(n,k)|$, we get that $|\ff|<|\mathcal C_3(n,k)|$.

From now on, we may assume that $\mathcal J$ is a dictatorship, say, $J=\{1\}$ and $\mathcal J=\{\{1\}\}$.\footnote{We note that the remaining part of the argument works for any $n>2k\ge 8$.} The next part of the proof will make use of the bipartite switching trick. We shall transform our family $\ff$ into another family (denoted by $\ff''$), which will satisfy $\tau(\ff'')=3$, $|\ff''|\ge |\ff|$ (with strict equality in case $\ff'$ is not isomorphic to $\ff$) and, moreover, $\ff''(\bar 1)$ will have covering number $2$ and a much cleaner structure.

To that end, take any $\mathcal M=\{M_1,\ldots, M_z\}\subset \ff'(\bar 1)$ such that $\tau (\mathcal M)=2$ and $\mm$ is minimal w.r.t. this property. Remark that $z\ge 3$ due to the fact that $\tau(\mm)=2$ and $\mm$ is intersecting. The next part of the proof borrows notations and ideas of Theorem~\ref{thmfull2}. For each $l\in[z]$, we can find $i_l$ as in \eqref{eqil}. W.l.o.g., assume that $\{i_1,\ldots, i_z\}=[2,z+1]$. Since $|\ff(\bar 1)|\le {n-5\choose k-4}$, we may apply the same exchange argument via $G'_i$ as in Theorem~\ref{thmfull2} and get a family $\ff'$, such that $\ff'(\bar 1) = \mm\cup \mathcal U$, where $\mathcal U$ contains only sets that contain $[2,z+1]$, and $\ff(1)$ consists of all sets intersecting all sets in $\ff'(\bar 1)$. 
We repeat the same exchange with any element contained in all but at one set from $\mathcal M$. At the end, we may assume that each $U\in\mathcal U$ contains $[2,t']$ for $t'\ge z+1$. As in the proof of Theorem~\ref{thmfull2}, we get $|\ff'|>|\ff|$ unless $\ff'$ is isomorphic to $\ff$ (recall that $k\ge C$ by assumption).

If $\mm$ is isomorphic to $\mathcal T_2(k)$ (cf. \eqref{deft2}), then the number of the elements contained in exactly two sets (all but one sets) is $k+1$, and thus we may immediately conclude that $\mathcal U=\emptyset$: no $k$-set can contain a subset of size $k+1$. Otherwise, $\mm$ is not isomorphic to $\mathcal T_2(k)$.\footnote{This is only needed for the uniqueness of the extremal family $\mathcal C_3(n,k)$, since some of the exchanges we shall perform below may not necessarily strictly increase the size. But this does not pose problems since we will eventually arrive at a family $\ff''$ with $\ff''(\bar 1)=\mm$, which we will show to have strictly smaller size than that of $\mathcal C_3(n,k)$.}

Let us show that we may assume that $\mathcal U$ is empty. As we have said, it is clear if $t'\ge k+2$. Otherwise, consider the family $\mm':=\{M\setminus [2,t']\ :\  M\in\mm\}$ and note that sets in $\mm'$ have size at least $1$. If there is no element $i'\in [t'+1,n]$ that is contained in at least $2$ sets of $\mm'$, then take two elements $i\in M'$ and $j\in M''$, where $t'+1\le i<j\le n$ and $M',M''$ are distinct sets in $\mathcal M'$, and perform the $(i,j)$-shift on $\ff'$. It is easy to see that $\tau(S_{ij}(\mathcal M))=2$, moreover, is
$S_{ij}(\ff)$ is intersecting. Thus, we may replace $\ff'$ with $S_{ij}(\ff')$ and $\mm$ with $S_{ij}(\mm)$.

Next, we can assume that there is an element in $i'\in [t'+1,n]$ that is contained in at least $2$ sets of $\mm'$. Now we may reuse the argument of Theorem~\ref{thmclass2} again. Consider the last exchange graph from the proof of Theorem~\ref{thmclass2}, i.e., $G(t',I)$, where $I$ is formed as follows. We include $i'$ in $I$, as well as one element from each of the sets from $\mm'$ that do not contain $i'$. We have $|I|\le z-1$. (Note that, as before, some elements we chose may coincide, making the last inequality strict.) In terms of the proof of Theorem~\ref{thmclass2}, we have $z'=|I|+1\le z\le t'-1=z''$. Moreover, by the choice of $I$, we have $|\mm\cap \mathcal P_b(t',I)| =\emptyset$. Therefore, we can perform the exchange operations as before for all possible choices of $I$, concluding that either all sets in $\mathcal U$ must contain $i'$, or they must contain some fixed set $M'\in \mathcal M'$. The latter is, however, impossible, since it would again imply that a $k$-element set from $\mathcal U$ contains a $(\ge k+1)$-element set $M'\cup [2,t']$.

Therefore, we may assume that all sets in $\mathcal U$ contain $i'$, and thus all contain $[2,t']\cup\{i'\}$. Finally, we may perform exactly the same exchange operations as in the proof of Theorem~\ref{thmfull2}, but with $[2,t']$ replaced by $[2,t']\cup \{i'\}$ in the definition $\mathcal P_a(t', I)$ and $\mathcal P_b(t',I)$ (denoted by $\mathcal P_a'(t', I)$ and $\mathcal P_b'(t',I)$, respectively). The reason it will work now is the extra fixed element in $\mathcal P_b'(t',I)$, which makes the number of fixed elements in $\mathcal P_b'(t',I)$  at least as big as in $\mathcal P_a'(t',I)$.

After the switches, we obtain a family $\ff''$, which is intersecting, satisfies $|\ff''|\ge |\ff|$ and $\ff''(\bar 1) = \mathcal M$. Moreover, it is easy to see that $\tau (\ff'')=3$. 

Finally, we need to show that, among all {\it minimal} families, the choice of $\mathcal T_2(k)$ is the unique optimal. But this is a direct application of the second part of Lemma~\ref{lemmin} with $s=k$ and $[2,n]$ playing the role of $[m]$. Note that $n>2k$, and $m\ge 2k$. The proof of Theorem~\ref{thmtau3} is complete.

\subsection{Proof of Theorem~\ref{thmbounddeg}}\label{sec52}
Let $C$ be a sufficiently large constant, which value shall be clear later.\footnote{As in the proof of Theorem~\ref{thmtau3}, we use a subscript ``$C$'' in the inequalities, in which we need that $C$ is sufficiently large.} The case of $k\le C$ follows from the original result of Frankl. In what follows, we assume that $k\ge C$. First of all, we remark that $\mathcal D_{3/7}$ satisfies the condition of the theorem, since $\Delta(\mathcal D_{3/7}) = \frac 37+O(\frac kn)<_C \frac 37+\epsilon.$

Take any family $\ff\subset {[n]\choose k}$ satisfying the requirements. Theorem~\ref{thmdf} implies that there exists a set $J$, $|J|\le c(5)$, and an intersecting family $\mathcal J^*\subset 2^{[J]}$, such that $|\ff \setminus \mathcal J|\le j(5){n-5\choose k-5}\le_C \epsilon{n-5\choose k-4}$, where $\mathcal J:=\{F\in {[n]\choose k}\ :\  F\cap J\in \mathcal J^*\}$. We have
\begin{equation}\label{eqbigsets}\big|\ff\setminus \{F\ :\ |F\cap J|\le 3\}\big|\le 2^{c(5)}{n-4\choose k-4}\le_C \frac{\epsilon}2{n-3\choose k-3}.\end{equation}

If $\mathcal J^*$ contains a set $R$ of size at most $2$, then, since $\mathcal J^*$ is intersecting, one of the elements of $R$, say $i$, satisfies $d_i(\ff\cap \mathcal J)\ge \frac 12 |\ff\cap \mathcal J|$, and thus $d_i(\ff)\ge \frac 12\big(|\ff|-|\ff\setminus \mathcal J|\big)\ge \frac 12\big( |\ff|-{n-5\choose k-4}\big)$. Given that $|\ff|\ge |\mathcal D_{3/7}|>{n-3\choose k-3}$, we get $d_i(\ff)\ge \frac 12|\ff|\big(1-\frac kn\big)$, which is bigger than $c|\ff|$ for  $n>k/\epsilon$, a contradiction. Thus, $\mathcal J^*$ contains only sets of size at least $3$.
 Let us put $\mathcal J_3^*:=\mathcal J^*\cap {[n]\choose 3}$ and $\mathcal J_3:=\big\{F\in {[n]\choose k}\ :\  F\cap J\in \mathcal J_3^*\big\}$.

 Recall that a \underline{fractional covering number} $\tau^*(\mathcal G)$ of a family $\G\subset 2^{[n]}$ is the minimum value of $\sum_{i=1}^n w_i$, where $0\le w_i\le 1$ are real numbers, such that $\sum_{i\in G}w_i\ge 1$ for each $G\in \G$. Note that if we allow $w_i$ to take integer values only, then we get back to the definition of covering number. Clearly, $\tau(\G)\ge \tau^*(\G)$.

Arguing as above, we have $\tau^*(\mathcal J^*_3)>2$. Indeed, it is easy to see that, for any family $\G$, $\Delta(\G)\ge |\G|/\tau^*(\G)$. Therefore, if $\tau(\mathcal J^*_3)\le 2$  then, of course, $\tau^*(\mathcal J_3\cap \ff)\le 2$ and so at least one element $i\in J$ satisfies $d_i(\ff\cap\mathcal J_3)\ge \frac 12\big|\ff\cap \mathcal J_3\big|$. Since the vast majority of sets from $\ff$ are actually from $\ff\cap \mathcal J_3$ (due to \eqref{eqbigsets}), we again get $d_i(\ff)>c|\ff|$, a contradiction.

F\"uredi \cite{Fur} showed that, for an intersecting family $\G\subset {[n]\choose p}$, we have  $\tau^*(\G)\le p-1+1/p$, and, moreover, if $\G$ is not a projective plane of order $p-1$, then $\tau^*(\G)\le p-1$. Thus, since $\tau^*(\mathcal J^*_3)>2$, it must be isomorphic to the Fano plane $\mathcal P$. In what follows, we assume that $\mathcal J^*_3 = \mathcal P$.

It is not difficult to check by simple case analysis that any set intersecting all sets in $\mathcal P$ must actually contain a set from $\mathcal P$. Thus, we get that $\mathcal J^*\subset \{S\subset J\ :\  P\subset S\text{ for some }P\in \mathcal P\}.$ We thus may also assume that $J = \bigcup_{P\in \mathcal P} P = [7]$ and $\mathcal J = \mathcal D_{3/7}$.

Summarizing the discussion above, we may assume that the junta that approximates $\ff$ is isomorphic to $\mathcal D_{3/7}$. The last step is to show that $\ff$ is actually isomorphic to $\mathcal D_{3/7}$.
Assume that $\ff':=\ff\setminus \mathcal J^*$ is non-empty. Then, for each $F\in \ff'$, there exists $P\in \mathcal P$ such that $F\cap P = \emptyset$. For each $P\in \mathcal P$, put
\begin{align*}\mathcal J(P)\ :=\ &\Big\{F\ :\  P\subset F\in {[n]\choose k}\Big\},\\
\mathcal \ff(P)\ :=\ &\mathcal J(P)\cap \ff,\\
\ff'(P)\ :=\ &\big\{F\ :\  F\in \ff',\ F\cap P=\emptyset \big\}.
\end{align*}
Then $\ff(P)$ and $\ff'(P)$ are cross-intersecting, moreover, the former family is $(k-3)$-uniform, while the second one is $k$-uniform and satisfies $|\ff'(P)|\le \epsilon{n-5\choose k-4}\le_C {n-7\choose k-4}$. Thus, we can apply \eqref{eqcreasy} with $k-3,k$ and $n-3$ playing the roles of $a,b$ and $n$, respectively, and get that replacing $\ff(P)\cup \ff'(P)$ with $\mathcal J(P)$ in $\ff$ will strictly increase the size of $\ff$ (given that it was not there in the first place). 
Repeating this for each $P\in\mathcal P$, we transform $\ff$ into $\mathcal D_{3/7}$. Moreover, if there was some change made in the process, then $|\ff|<|\mathcal D_{3/7}|$. Thus, we conclude that $\ff$ is isomorphic to $\mathcal D_{3/7}$.

\section{Conclusion}\label{sec6}

In this paper, we found the largest intersecting family $\ff\subset {[n]\choose k}$ with $\tau(\ff)\ge 3$, provided $n>Ck$ for some absolute constant $C$. Actually, we proved more than that. Provided that $\gamma(\ff)\le {n-5\choose k-4}$ and $\tau(\ff)\ge 3$, we showed that $|\ff|\le |\mathcal C_3(n,k)|$, with equality only if $\ff$ is isomorphic to $\mathcal C_3(n,k)$, for any $n>2k\ge 8$. With some effort, we can replace the condition on diversity by $\gamma(\ff)< {n-4\choose k-3}$. However, since we use juntas result to show that the diversity of any extremal $\ff$ must be small, the condition $n>Ck$, $k>C$ is necessary for our approach. This motivates the following problem.
\begin{pro}
  Determine $c(n,k,3)$ for all $n>2k$.
\end{pro}
We believe that the family $\mathcal C_3(n,k)$ should be extremal for all $n>2k\ge 10$.

What can one say about $c(n,k,t)$ for $t\ge 4$? Using juntas, one can similarly show that the largest family must have very small diversity (say, at most ${n-2t\choose k-2t}$) for $n>C(t)k$. Using similar proof logic, this  allows to deduce the result of \cite{FOT1} for $n>Ck^4$.  The main difficulty (at least for $n>Ck^t$) for $t\ge 5$ lies in the following problem.

\begin{pro}\label{propbol} Given an intersecting family $\ff$ of $k$-sets with $\tau(\ff)=t$, what is the maximum number of hitting sets of size $t$ it may have?
\end{pro}

Finally, we state the following question:
\begin{pro}\label{propbol} What is the maximum of $\gamma(\ff)$ for intersecting $\ff\subset {[n]\choose k}?$
\end{pro}

I showed in \cite{Kup21} that $\gamma(\ff)\le {n-3\choose k-2}$ for $n>Ck$ with some absolute constant $C$. In particular, this makes Theorem~\ref{thmfull1} complete: in that range, it covers all possible values of diversity. Following Frankl, I conjectured that the same bound should hold for $n\ge 3k$, however, a counterexample was provided by Huang \cite{Hua} for $n\le (2+\sqrt 3)k$. I still believe that $\gamma(\ff)\le {n-3\choose k-2}$ should hold for moderately large $n$, say, for $n\ge 5k$.  The case of smaller $n$ is also very interesting.\\

{\sc Acknowledgements:} I would like to thank Peter Frankl for introducing me to the area and for numerous interesting discussions we had on the topic.

\end{document}